\documentclass[11pt]{article}
\usepackage[utf8]{inputenc}
\usepackage[english]{babel}
\usepackage{amsmath,amsfonts,amssymb, amsthm}
\usepackage{tikz-cd} 
\usepackage{mathtools}          

\usepackage[h margin=1.3 in, v margin=1.2 in]{geometry}
\usepackage{enumerate}
\usepackage{enumitem}
\usepackage{mathrsfs}
\usepackage[colorlinks=true, allcolors=blue]{hyperref}
\usepackage{graphicx}
\usepackage{tikz-cd}
\usepackage{tensor}
\usepackage{cancel}
\usepackage{comment}
\usepackage{braket}

\theoremstyle{plain}

\newtheorem{fact}{Fact}[section]
\newtheorem{lemma}[fact]{Lemma}
\newtheorem{theorem}[fact]{Theorem}
\newtheorem{proposition}[fact]{Proposition}
\newtheorem{corollary}[fact]{Corollary}

\theoremstyle{definition}
\newtheorem{example}[fact]{Example}
\newtheorem{remark}[fact]{Remark}
\newtheorem{definition}[fact]{Definition}

\newtheorem*{question*}{Question}

\makeatletter
\newcommand{\superimpose}[2]{%
  {\ooalign{$#1\@firstoftwo#2$\cr\hfil$#1\@secondoftwo#2$\hfil\cr}}}
\makeatother

\usepackage[backend=biber,style=alphabetic]{biblatex}
\makeatletter
\let\c@equation\c@thm
\makeatother
\graphicspath{{images/}}

\newtheorem*{theorem*}{Theorem}
\theoremstyle{definition}

\newcommand{\bb}[1]{\mathbb{#1}}

\DeclareMathOperator{\Aut}{Aut}

\newcommand{\C}{\mathbb{C}}
\makeatletter
\newcommand{\extp}{\@ifnextchar^\@extp{\@extp^{\,}}}
\def\@extp^#1{\mathop{\bigwedge\nolimits^{\!#1}}}
\makeatother
\DeclareMathOperator{\csm}{csm}
\DeclareMathOperator{\ssm}{ssm}

\let\emptyset\varnothing

\title{Higher order double point formulas via SSM–Thom polynomials}
\author{Reese Lance}
\date{}

\begin{document}
\maketitle

\begin{abstract}
We study the geometry of double point loci of maps $F:M\to N$ of complex manifolds through the lens of Segre–Schwartz–MacPherson (SSM) classes. Classical double point formulas express the fundamental class of the closure of the double point locus of $F$ in terms of global invariants of source and target spaces, as well as $F$. In this paper we extend these results by computing a one-parameter cohomological deformation of the double point formula given by the SSM class. We compute the SSM class of the double point locus in a large cohomological degree range. The leading term in our new formulas recovers the classical double point formula of Fulton and Laksov, while higher-degree terms provide explicit universal corrections. Our approach uses interpolation techniques for SSM–Thom polynomials of multisingularities, recently developed by Koncki, Nekarda, Ohmoto and Rim\'anyi. We also compute SSM–Thom polynomials for the singularities $A_0$ and $A_1$ in the same range. As an application, we show how the deformed formulas yield refined geometric information about those singularity loci through a theorem of Aluffi and Ohmoto, 
including constraints on when such loci can arise as complete intersections. 
\end{abstract}

\tableofcontents

\section{Introduction}
\subsection{Multiple point formulas (The classical setting)}
The main theorem of this paper extends a family of theorems from classical enumerative geometry called the \textit{double point formulas}, the first non-trivial cases of the more general \textit{multiple point formulas}. Our main theorem can be interpreted as a 1-parameter deformation of the double point formula. These classical formulas appear in many different settings, but all answer the following elementary question: Given a map $F: M\to N$ in some appropriate category (e.g. smooth, algebraic), let $\Delta(F) \subset M$ denote the set of points $m\in M$ for which $F^{-1}(F(m))$ has exactly $r$ points. The multiple point locus $\Delta(F)$ records the geometry of the self-intersection behavior of $F$. When $r=2$, $\Delta(F)$ is known as the \textit{double point set} or \textit{double point locus} of $F$.

Under suitable hypothesis, $\overline{\Delta(F)}$ is an object of the same category, and admits a fundamental class in an appropriate cohomology theory or Chow ring. The fundamental question is whether the topology of the multiple point locus can be expressed in terms of global data.

\begin{question*}
    Can $\left[\overline{\Delta(F)}\right]\in H^*(M)$ be expressed in terms of simple invariants of $M,N$, and $F$? 
\end{question*}

For this question to have a positive answer, it must be assumed that $F$ has some genericity or transversality properties, which also depend on the specific setting. 

This question was investigated in the smooth category when $F$ is a smooth immersion, by a collection of authors who found recursive formulas for the fundamental class in terms of $r$, \cite{LS}, \cite{herbert}, \cite{Ronga1980}. Crucially the work of Ronga introduced the first version of the \textit{double point scheme}, a modification of the double point set which would later be used by Laksov in the algebro-geometric setting. Separately, Kleiman invented methods for the computation of multiple point formulas when points are allowed to be ``infinitely close''. \cite{Kleiman} established the process of iteration, and \cite{Kleiman-hilb} showed a new method using Hilbert schemes.

The case of $r=2$, the double point formulas, was initially studied by Todd and Whitney in the differentiable setting in the 1940's. It was studied in an algebro-geometric context by Fulton and Laksov who found concrete formulas using their technology of residual intersection theory. \cite{laksov} first utilized Ronga's notion of double point scheme to obtain the first algebro-geometric double point formulas in a restricted setting, which was later expanded on in \cite{fulton-laksov} and \cite{FultonIT}. We will recall their description here.

Let $F: M \to N$ be a morphism of nonsingular, complex varieties of dimensions $n,m$, with $M$ complete and $n\geq m$. Let $\Delta(F)$ be the double point set. Under some genericity conditions on $F$, they prove:

\begin{theorem}[Double point formula]
    In the Chow ring $\mathcal A^{n-m}(M)$,
    \[
    \big[\Delta(F)\big] = F^*F_*(1)- c_{n-m}\big(F^*(TN)-TM\big)
    \]
    where $c_{n-m}$ is the $(n-m)$th Chern class of the virtual bundle $F^*(TN)-TM$ on $M$. 
\end{theorem}

\begin{remark}
This remarkable formula shows that the topology of the double point locus is surprisingly insensitive to the particular setting. It does not depend on both dimensions $m$ and $n$ independently, only their difference $n-m$. Further, when $M,N$ are fixed, the dependence on $F$ is similarly mild, depending only on the homotopy class of $F$. In this sense, the double point formula is \textit{universal}, so this topic and its extensions fit into the general study of ``universal polynomials expressing the geometry of singularity loci''. 
\end{remark}

\begin{example}[\cite{FultonIT}]
The double point formula recovers a classical enumerative relation between a nodal plane curve $C \subset \bb P^2$ and its normalization, $X$. Let $d$ be the number of node singularities of $X$, $n$ be the degree of $C$, and $g$ be the genus of $X$. Then the double point formula applied to the canonical morphism  $F: X \to \bb P^2$ yields
\[
2d=n^2-3n+2-2g.
\]
\end{example}

Extensions and generalizations of the double point formula are thus likely to lead to further enumerative insights. 

\subsection{Extension by Segre-Schwartz-MacPherson classes}

In this paper we study a one parameter deformation of the cohomological fundamental class known as the Segre-Schwartz-MacPherson (SSM) class of a variety. The SSM class is the unique extensions of the fundamental class satisfying functoriality properties. Formulas for the SSM class will thus determine higher order terms deforming classical formulas for cohomological fundamental classes. 

For a (possibly singular) locally closed variety, $\Sigma$, embedded in a smooth variety $M$, the SSM class is an inhomogeneous class in the cohomology or Chow ring of $M$, whose lowest degree term is the cohomological fundamental class of the closure

\[
\ssm\big(\Sigma\subset M\big) = \left[\overline\Sigma\subset M\right] + h.o.t.
\]
We recall the details of the SSM class in Section \S\ref{subsection:CSM-SSM-intro}. 

In the case $\Sigma=\Delta(F)$, we expect the higher order terms to be universal in the same sense as in the previous section. The main theorem of the present paper provides an explicit computation of the SSM class of $\Delta(F)$ in a large cohomological degree and dimension range, see Theorem (\ref{SSM of A02}), for a certain class of maps called $\mathcal C_\ell$-maps, which have controlled singularity types.

To fix notation, let $m,n\geq 0$ and $\ell:=n-m$. For $F: M \to N$ a map of nonsingular, complex varieties of dimensions $m,n$, denote 
\[
c_i(F):=\frac{c(F^*TN)}{c(TM)}\Bigg|_{\deg i},\quad\quad s_\lambda(F) = F^*F_*\left(\prod_ic_{\lambda_i}\right).
\]
where $c(V)$ denotes the total Chern class of the vector bundle $V$. Then the classical double point formula reads (with the convention that $c_\emptyset=1$)
\[
\left[\overline{\Delta(F)}\right] = s_\emptyset(F)-c_\ell(F)
\]
for \textit{any} $F$ meeting a genericity condition. 

The new Theorem (\ref{SSM of A02}) extends the classical double point formula in a degree range depending on $\ell$. For illustration, we list the first few cases below (suppressing the dependence on $F$). In each case, the leading term recovers the classical formula, while subsequent terms represent higher-order corrections.

\begin{example}\label{ex: SSM A02}
\begin{align*}
\ell=1 : \quad\quad \ssm\big(\Delta(F)\big) &= (s_\emptyset-c_1)+\dots\\\\
\ell=2 : \quad\quad \ssm\big(\Delta(F)\big) &= (s_\emptyset-c_2) + (c_1c_2+2c_3-s_1) + \dots\\\\
\ell=3 : \quad\quad \ssm\big(\Delta(F)\big) &= (s_\emptyset-c_3) + (c_1c_3+3c_4-s_1) + (-c_1^2c_3-4c_1c_4+c_2c_3-6c_5+s_{11}-s_2)+\dots \\\\
\ell=4 : \quad\quad \ssm\big(\Delta(F)\big) &= (s_\emptyset-c_4) + (c_1c_4+4c_5-s_1) + (-c_1^2c_4-5c_1c_5+c_2c_4-10c_6+s_{11}-s_2) \\
&+(c_{1}^{3}c_{{4}}+6c_{1}^{2}c_{5}-2c_{1}c_{2}c_{4}+
15c_{{1}}c_{{6}}-6c_{{2}}c_{{5}}+c_{{3}}c_{{4}}+20c_{{7}}-s_{111} +2s_{21}-s_3) +\dots\\\\
 &\vdots
\end{align*}
\end{example}
The method to prove Theorem (\ref{SSM of A02}) and its corollaries is called \textit{interpolation}, and fits into the study of SSM-Thom polynomials, which we will recall in Section \S\ref{subsection: ssm-polys}. These are universal power series in $c$ and $s$ variables which determine the SSM classes of multisingularities for all $\mathcal C_\ell$ maps, which are our main object of study. In the study of singularities, the double point set is interpreted as a ``multisingularity type'' called $A_0A_0$, or $A_0^2$. It is important to note that there is a theoretical bound in cohomological degree called the Mather bound, approximately $6\ell+7$, at which the interpolation technique no longer applies. So while it is possible to compute even higher order terms than what we achieve here, the interpolation method in its present form cannot prove closed formulas for the full power series. 

The rest of the paper is organized with the SSM-Thom polynomial as the focus, of which ssm$(A_0^2)$ will be our main particular case. Therefore we will discuss the theory of multisingularities. We also compute the SSM-Thom polynomial of two other singularities, named $A_0$ and $A_1$, in the same degree and dimension range. In a sense, these three together are the most frequently occurring singularities. 
 
\subsection{Higher order geometry of double point loci}
As an application, we show how the new formulas obtained for SSM-Thom polynomials can be used to extract fine geometric information about singularity loci. The main tool used is Theorem (\ref{linear sections}) of Aluffi and Ohmoto, which shows that the SSM class of a variety determines the Euler characteristics of all general linear sections of that variety. Further, we show that in a certain setting of maps of projective spaces, using the first non-trivial deformation of the SSM-Thom polynomial of $A_0^2$, the $A_0^2$ locus of a $\mathcal C_\ell$ map cannot be a complete intersection. 

\subsection{Acknowledgements}
We would like to thank Rich\'ard Rim\'anyi for many helpful discussions. 
\section{Setup}
\subsection{Chern-Schwartz-MacPherson/Segre-Schwartz-MacPherson Classes}\label{subsection:CSM-SSM-intro}

We consider possibly singular, possibly not closed varieties $\Sigma$ which are embedded in an ambient compact, complex manifold, $\Sigma \subset M$. The ground field will always be $\C$, and cohomology coefficients will be taken to be rational, unless otherwise stated.

The fundamental cohomology class of a subvariety $[\overline\Sigma\subset M]\in H^{2\cdot\text{codim}(\Sigma)}(M)$ captures some coarse geometric information about the embedded subvariety. This class admits a one-parameter deformation in cohomology, called the CSM class. The CSM class is an inhomogeneous class whose lowest order term is the fundamental class: 
\[
\csm(\Sigma\subset M) = [\overline{\Sigma}\subset M]+h.o.t. \quad \quad \in H^*(M).
\]
We also consider the \textit{Segre-Schwartz-MacPherson} (SSM) class which differs by a multiplicative factor
\[
\ssm(\Sigma\subset M) = \frac{\csm(\Sigma\subset M)}{c(TM)}.
\]
Since the factor is explicit ($M$ is smooth), these different versions carry equivalent information, and it is a matter of convenience or preference which version to work with. We prefer to work with SSM class, which behaves well with respect to pull-back, as opposed to CSM which behaves well with respect to push-forward. 

If there is a group $G$ acting on $M$ which preserves $\Sigma$, there are equivariant versions of CSM and SSM classes in equivariant cohomology, established in \cite{OHMOTO_2006}. In this case the equivariant cohomology rings are infinite dimensional, and the degree of the SSM class is unbounded, so lives in a completion $H^{**}_G(M)$. The existence of the (non-equivariant) CSM and SSM classes was conjectured by Deligne and Grothendieck and constructed independently by MacPherson \cite{macpherson} and Schwartz \cite{schwartz}, and the equivariant version (for our setting) was proven by Ohmoto \cite{ohm24}. In MacPherson's formulation, the CSM classes are constructed as the unique classes satisfying a collection of axioms. 
\begin{enumerate}
\item (Push-Forwards): CSM classes behave predictably under push-forward: If $f:\Sigma \to Y$ is an equivariant, proper morphism, then
\[
f_*(\csm(\Sigma)) =\sum_{n\in\mathbb Z} n\cdot\csm(Y_n),
\]
where
\[
Y_n = \big\{y \in Y\ |\ \chi\big( f^{-1}(y)\big)=n\big\}.
\]

\item  (Normalization): If $\iota: \Sigma\subset M$ is a closed embedding, of smooth manifolds, then 
\[
\csm(\Sigma\subset M)=i_*c(T\Sigma).
\]

\item (Additivity): 
\[
\csm(X\sqcup Y) = \csm(X) + \csm(Y) - \csm( X \cap Y)
\]
if $X,Y$ are locally closed. 
\end{enumerate}

These axioms can be stated in a more general setting, in which CSM classes are defined for all constructible functions on $M$, and our notion of CSM class of a subvariety here corresponds to the indicator function on that subvariety: $\csm(\Sigma)\equiv \csm(\textbf{1}_\Sigma)$. In that setting, the CSM class is defined as the unique natural transformation from the functor of constructible functions to the functor of Borel-Moore homology, satisfying a similar set of axioms. That there exists such a unique natural transformation is the content of the previously cited papers. The definition and set of axioms we write here is a slightly specialized version, which is suitable for our purposes. 

\subsection{Singularities of Maps}

Standard reference for the material of this section is the textbook \cite{mn20}. 

This discussion necessarily involves many infinite dimensional varieties and groups, which are often treated as if they are finite dimensional. To be precise, one must replace these infinite dimensional spaces with ``approximation spaces'', finite dimensional spaces which have enough properties in common with the original. This is a standard technique in equivariant geometry, and is discussed in general settings in \cite{AF}. For this case, it corresponds to replacing all the infinite dimensional spaces of germs with the finite dimensional spaces of $k$-jets of germs, for very large $k$. We will omit this notation and not discuss this angle further, since there is no dependence on $k$, provided it is large enough.

Let $n\geq m$ and set $\varepsilon^\circ(m,n)$ as the vector space of holomorphic germs $\C^m\to \C^n$ sending $0\mapsto 0$. On this space we consider two actions: Let $\mathcal A :=\Aut(\C^m,0)\times \Aut(\C^n,0)$ be the \textit{right-left} group of biholomorphic germs (still sending $0\mapsto 0$), which acts on $\varepsilon^\circ(m,n)$ by conjugation on source and target. Also define $\mathcal K\supset \mathcal A$ to be the group of holomorphic germs of $\Aut(\C^m\times \C^n,0)$ of the form 
\[
\Phi(x,y) = \big(\phi(x),\psi(x,y)\big)
\]
with $\psi(x,0)=0$. The group $\mathcal K$ is called the \textit{contact group}, and its action on $\epsilon^\circ(m,n)$ is induced by its natural action on the graph of a germ. The action is such that two germs are in the same $\mathcal K$-orbit iff their graphs have the same contact with the graph of $y=0$.

A \textit{right-left} singularity (resp. \textit{contact} singularity) is a subvariety of $\varepsilon^\circ(m,n)$ invariant under the action of $\mathcal A$ (resp. $\mathcal K)$. There are many types of right-left singularities, but contact singularities enjoy several properties making them especially suitable for study. There is a convenient characterization of contact singularities in terms of local algebras. 

If $f\in \varepsilon^\circ(m,n)$, the \textit{local algebra} of $f=(f_1,\dots,f_n)$, $Q_f$, is the quotient 
\[
Q_f:=\C[[x_1,\dots,x_m]]/(f_1,\dots,f_n).
\]
This quotient will not always be finite dimensional, but for this paper we will restrict attention to only those germs whose local algebra is finite dimensional. 
A theorem of John Mather says that two germs are $\mathcal K$-equivalent iff their local algebras are isomorphic, \cite{mather}. 

Therefore contact singularities are labeled by isomorphism classes of finite dimensional, commutative, local algebras. We will often refer to contact singularities by solely their local algebras. Though this is is a convenient labeling method, the theory of finite dimensional, commutative, local algebras is still complex, and determining whether two such algebras are isomorphic is not trivial.

\begin{example}
Some families of local algebras. 
\begin{enumerate}
\item $A_i:=\C[x]/(x^{i+1})$
\item $I_{a,b} :=\C[x,y]/(xy,x^a+y^b)$
\item $C_\lambda:=\C[x,y,z]/(x^2-\lambda yz,y^2-\lambda xz,z^2-\lambda xy)$.
\end{enumerate}

\end{example}

Since $\mathcal A\subset \mathcal K$, given a contact singularity $\eta\subset \varepsilon^\circ(m,n)$, and a map of compact, complex manifolds $F: M \to N$ , we can define the $\eta$-points of $F$, $\eta(F)\subset M$, as those points $p \in M$ such that the germ of $F$ at $p$, $[F]_p$, belongs to $\eta$. $\mathcal A$-invariance guarantees this notion will be invariant under change of coordinate charts around $p$ and $F(p)$. The singularities considered thus far are referred to as \textit{monosingularities.}

In contrast, a (contact) \textit{multisingularity} is a concatenation of (contact) monosingularities, $\underline{\eta}:=\eta_1\cdots\eta_r$. The $\underline{\eta}$-points of $F$ may similarly be defined: An $\underline{\eta}$ point of $F$ is a set of $r$ points of $M$, $x_1,\dots,x_r$, with $F(x_1)=\dots =F(x_r)$, and such that $F$ has a singularity of type $\eta_1$ at $x_1$, and that the remaining, unordered points are of unordered singularity types $\eta_2,\dots,\eta_r$. In our convention, a multisingularity has a distinguished first component, $\eta_1$, while the rest of the components are unordered. For this reason, a multisingularity may be thought of as a multiset with a distinguished element. 

\begin{remark}
By definition, monosingularities are subvarieties of the affine space $\varepsilon^\circ(m,n)$, while multisingularities must live in some complicated form of a configuration space, see \cite{ohm24}. This fact is the main reason that the theory of multisingularities is significantly more delicate than that of monosingularities. In particular, the equivariant cohomology of the classifying space of $\varepsilon^\circ(m,n)$ is well understood in terms of generators, while the equivariant cohomology of candidate spaces for multisingularities is less understood, though this is an area of active research.
\end{remark}

\subsubsection{Mather singularities}
Fix $\ell\geq 1$, and denote the \textit{Mather bound}, $M(\ell)$, as
\[
M(\ell) = \begin{cases} 6\ell+8 & \ell\leq 3\\ 6\ell+7& \text{else}\end{cases}.
\]
According to \cite{mather}, for any $d\leq M(\ell)$, the list of local algebras arising from contact monosingularities of germs $(\C^m,0)\to (\C^{m+\ell},0)$, of codimension $\leq d$ is finite and when $m$ is suitably large, depends only on $\ell$. There are 55 local algebras arising from monosingularities for some $\ell$, which are called Mather algebras. The corresponding singularities are called Mather (mono)singularities. 

Similarly, multisingularities with codimension at most $M(\ell)$ are called \textit{Mather multisingularities}, and the list of such is expected to have analogous properties, though there are many more Mather multisingularities. For the degree bounds we will consider, the list of Mather multisingularities is known. 

\subsection{SSM-Thom polynomials}\label{subsection: ssm-polys}
There are many modern review papers on this topic now, for example see \cite{tpprimer, ohmoverview}. The foundations of SSM-Thom polynomials and the corresponding Thom principle in this setting were introduced in \cite{ohmoto2016}.

We are interested in the SSM classes of singularities, considered as subvarieties of $\epsilon^\circ(m,n)$ and extracting geometric information about singularity loci of maps using the ``Thom Principle'', to be described below. In this context, the SSM class of a singularity is referred to as the \textit{SSM-Thom polynomial}\footnote{This is an ``abuse-of-naming'', as the SSM class is actually a power series.}, denoted $\ssm\big(\eta\subset\epsilon^\circ(m,n)\big)$: the lowest non-zero degree of the SSM class is an older, classical object called the Thom polynomial. SSM-Thom polynomials are thus one parameter deformations of Thom polynomials.

\begin{remark}
In this paper, we only consider so-called \textit{source SSM-Thom polynomials}, describing the SSM class of the subset of the domain where $f$ has a prescribed singularity type. In the literature, the \textit{target SSM-Thom polynomials} are also considered, describing the SSM-class of the image of the source singularity locus. Source SSM-Thom polynomials are more delicate: For example, later in this section, we explain that the source SSM-Thom polynomial of a contact multisingularity depends on a two sets of variables $c_i,s_\lambda$, while the target SSM-Thom polynomial depends only on the $s_\lambda$. Additionally, the source singularity itself is a multiset with a distinguished element, while the target multisingularity may be described as simply a multiset. In principle, the target SSM-Thom polynomial can be obtained from the source SSM-Thom polynomial. For the rest of this paper, we study only to the source version. 
\end{remark}

\subsubsection{Stability of maps}
SSM-Thom polynomials encode geometric information about maps which are stable under perturbations. The following definitions of various notions of stability are found in \cite{mn20}, until otherwise mentioned (The authors present many equivalent notions of stability of a map of manifolds, here we recall just one version).

Let $F: M\to N$ be a smooth map of manifolds. $F$ is \textit{stable} if, for any open neighborhood, $U \ni F$ in the Whitney topology of $C^\infty(M,N)$, every $G \in U$ is related to $F$ by diffeomorphisms of the source and target
\[
G=\phi \circ F\circ \psi^{-1}.
\]
Due to Thom transversality (over the reals), the family of stable maps $M\to N$ is dense in $C^\infty(M,N)$ when considering real manifolds, since every map can be perturbed to a stable one. However over $\C$, there is no such notion of perturbation, for example because there are no holomorphic partitions of unity. Instead we consider a variant of the notion of stability, the class $\mathcal C_\ell$, for which we need some intermediate definitions. Let $n\geq m$ and $\ell:=n-m.$

\begin{definition}\label{defn-stable}
    Let $F: (\C^m,S)\to (\C^n,0)$ be a germ sending $S \mapsto 0$. For $d\in \bb N$, a \textit{$d$-unfolding} of $F$ is a germ $\textbf{F}: (\C^m\times \C^d,S\times \{0\})\to (\C^n\times \C^d,0)$ of the form $\textbf{F} = \big(\tilde F(x,u),u\big)$ where $\tilde F(x,0)=F$.
    
    In other words, a $d$-unfolding is a germ with an additional $d$ variables, $u$, such that setting all $u=0$ recovers the original germ. 

    A germ is called \textit{locally stable} if all of its unfoldings are equivalent to an unfolding of the form $G=g\times \text{Id}_{\C^d}$.
\end{definition}

\begin{definition}
    A map of manifolds $F: M\to N$ is called \textit{target-locally stable} if, for every $y \in N$, 1) $F^{-1}(y)$ is finite, and 2) every induced germ
    \[
    [F]_{F^{-1}(y)}: \big(\C^m,F^{-1}(y)\big) \to (\C^n,0)
    \]
    is locally stable, in the sense of Definition \ref{defn-stable}. 
\end{definition}

Finally, the following class is a specialization of the definition of $\tilde{\mathcal C}^{Ma,k}_\ell$ from \cite{rimanyi}, to the setting with which we will eventually be concerned.

\begin{definition}
    For $\ell\geq 1$, and a pair of compact, complex manifolds $M,N$ with $\dim(M)+ \ell=\dim(N)$, let $\mathcal C_\ell$ be the subset of $C^\infty(M,N)$ consisting of target-locally stable maps which are finite morphisms. Maps belonging to $\mathcal C_\ell$ are referred to as $\mathcal C_\ell$-maps, when a choice of $M,N$ is understood. 
\end{definition}

$\mathcal C_\ell$-maps are thus proper, so have well-defined push-forwards, and have only Mather multisingularities. These properties are necessary for the technique we will employ. 

It is interesting to note that, though we study source SSM-Thom polynomials, the correct notion of stability still manifests in the target. 

\subsubsection{Thom Principle}
When $\eta$ is a monosingularity, 
\begin{align*}
\ssm_{\mathcal K}\big(\eta\subset\varepsilon^\circ(m,n)\big)\in H_{\mathcal K}^{**}\big(\epsilon^\circ(m,n)\big)\cong \mathbb{Q}[[a_1,\dots,a_m,b_1,\dots,b_n]]
\end{align*}
where $a_i,b_i$ are equivariant Chern classes associated to $GL_m(\C),GL_n(\C)$. Thus equivariant SSM-Thom polynomials are power series in the variables $a_i,b_i$. When $\eta$ is a monosingularity, not necessarily of contact type, the equivariant SSM class can depend on $a_i,b_i$ independently. For contact monosingularities, it is known that the SSM class depends only on the quotient variables $c_i$, where
$$
c_i:=\frac{1+b_1+b_2+\dots+b_m}{1+a_1+a_2+\dots+a_n}\Bigg|_{\text{degree }i}
$$
and depends only on $m,n$ through the quantity $n-m:=\ell$. That is to say: For a choice of $m,n$, and any $s\geq 0$, the power series in the quotient variables representing $\ssm_{\mathcal K}(\eta\subset\epsilon^\circ(m,n))$, and the power series representing $\ssm_{\mathcal K}(\eta\subset\epsilon^\circ(m+s,n+s))$, are identical, up to consistent re-labeling of indices.

This fact is now known in various contexts as the Thom-Damon-Ronga theorem, and was first proven for \textit{Thom polynomials} of contact monosingularities in \cite{damon, ronga72}, see also \cite{thom}.

This is what it means to ``compute'' an SSM-Thom polynomial, to present it as a power series in $\underline{c}$ with parameter $\ell$. We denote the SSM-Thom polynomial of the singularity $\eta$, for maps of relative dimension $\ell$, by $
\ssm_{\mathcal K}(\eta,\ell)$. This phenomena, the dependence only on the parameter $\ell=n-m$, mirrors the behavior of Mather singularities. 

The motivation to find these power series representations comes from the following, which may be referred to as the Thom-Kazarian principle, or Thom principle depending on the context (we will use the latter for the remainder of this paper): let $\eta$ be a contact monosingularity, and express the equivariant SSM class in a power series, $P$,
\[
\ssm_\mathcal K\big(\eta\subset \epsilon^{\circ}(m,n)\big) \equiv \ssm_{\mathcal K}(\eta,\ell) = P(\underline{c}) \in H^{**}_{\mathcal K}\big(\varepsilon^\circ(m,n)\big).
\]
Then for a stable map of complex manifolds $F: M \to N$, the power series $P$ determines the (non-equivariant) SSM class of the $\eta$ points of $F$,

\begin{equation}\label{univ-mono}
 \ssm\big(\eta(F)\subset M\big) = P(\underline{c})\Big|_{\underline{c}=\underline{c}(F)} \in H^{*}(M)
\end{equation}
where $\underline{c}(F)$ are the quotient Chern classes of $F$,
\[
\underline{c}(F):=\frac{F^*c(TN)}{c(TM)}.
\]
For this reason, SSM-Thom polynomials are referred to as ``universal polynomials'', since they determine the SSM classes of the singularity loci of \textit{any} stable map, with only mild dependence on the choice of map, through the characteristic classes $\underline{c}(F)$. That equality
(\ref{univ-mono}) holds for any stable map $F$ is the essence of the Thom principle. For monosingularities, this principle is essentially a simple deduction from the definition of SSM-Thom polynomial, and the theory of classifying spaces and universal bundles. Crucially, for a specified singluarity $\eta$, and a $\mathcal C_\ell$-map $F: X \to Y$, we consider the $k$th jet bundle $J^k(X)\to X$, for some very large $k$, where a natural notion of $\eta(F)$ lives. As $J^k(X)\to X$ is a $\mathcal A$-bundle, there is a universal bundle, $E\mathcal A \to B\mathcal A$, and a classifying map $\kappa: X\to B\mathcal A$. We perform calculations in $H^{**}(B\mathcal A)$, then pull them back along $\kappa$, which assigns $\underline{c}=c(F)$ (the correct way to obtain the map $\kappa$ from the initial setup actually depends on $F$). 

When $\underline{\eta}$ is a contact \textit{multi}singularity, there exists an analogous power series $P$, now in two sets of variables, $\underline{c}$ and $\underline{s}$, the latter of which being indexed by partitions, and parameter $\ell$, satisfying the property that for any $\mathcal C_\ell$-map, $F: M \to N$, 

\begin{equation}\label{univ relat SSM multi}
 \ssm\big(\underline{\eta}(F)\subset M\big) = P(\underline{c},\underline{s})\Big|_{\underline{c}=\underline{c}(F),\underline{s}=\underline{s}(F)}
\end{equation}
where $\underline{c}(F)$ are same as above and $\underline{s}(F)$ are the (source) \textit{Landweber-Novikov} classes, 
\[
s_\lambda(F):=F^*F_*(c_\lambda(F))
\]
e.g.
\[
s_{1,2,2} \equiv F^*F_*(c_1c_2^2).
\]
To establish the existence and uniqueness of universal formulas for {\em multi}singularity loci, e.g. their SSM classes, is not an obvious extension from the {\em mono}singularity case, see \cite{Kazarian_2003, Kazarian_2006, ohm24, rimanyi}. Yet, for suitable maps and in suitable dimension and cohomological degree range, the existence, uniqueness, and structure theorems are established in \cite{rimanyi}. Our settings falls under this category. Hence for the purposes of this paper, we will treat $\ssm(\underline{\eta},\ell)$ as a mere symbol representing a power series in variables $\underline{c},\underline{s}$ with parameter $\ell$ such that, for any $\mathcal C_\ell$-map $F: M\to N$, the relation (\ref{univ relat SSM multi}) holds.

\section{Interpolation for SSM-Thom Polynomials}

A technique called \textit{interpolation}, initially introduced in \cite{Rim2001}, can be used to compute coefficients of SSM-Thom polynomials. This technique is inspired by an analogous technique in the study of stable envelopes, specifically the axiomatic description introduced by Maulik-Okounkov \cite{MO}. In some restricted settings, the two constructions of CSM class and stable envelope were shown to coincide, see \cite{FR18, RV18}, which served as the inspiration to consider interpolation techniques for CSM/SSM classes in general. Most recently, interpolation for SSM-Thom polynomials has been proven for contact monosingularites in \cite{rimanyi2025interpolationcharacterizationhigherthom}, and for contact multisingularity loci of $\mathcal C_\ell$ maps in \cite{rimanyi}. We will utilize a form of the interpolation theorem, Proposition (8.7) from \cite{rimanyi}, and some notation from \cite{rimanyi2025interpolationcharacterizationhigherthom}. 

The Thom principle states that the SSM-Thom polynomial of the abstract variety representing the singularity type is equal to the SSM class of the singularity locus of any $\mathcal C_\ell$-map, upon evaluation of the abstract variables. The idea of interpolation is that SSM-Thom polynomials are actually \textit{determined}, up to a chosen degree, by finitely many of these relations, thus making their computation, in principle, algorithmic. 

Let $\ell\geq1$ and choose a degree $d\leq M(\ell)$, where $M(\ell)$ is the Mather bound, and let $Q_1,\dots,Q_r$ enumerate the list of Mather singularities with codimension lesser or equal to $d$. For $i=1,2,\dots,r$, denote 

\begin{enumerate}
\item $p_i: (\C^{s_i},0)\to (\C^{s_i+\ell},0)$ as a prototype for $Q_i$ and $\ell$. The theory of genotypes and prototypes is also described in \cite{mather}. Loosely, a prototype for the local algebra of a contact singularity is a germ whose local algebra is $Q_i$, and which is stable under infinitesimal deformations. There is an algorithm for writing down prototypes called \textit{(mini)-versal unfolding}.

\item $\bb T_i$ the maximal torus of the maximal compact symmetry group, $G_i$, of $p_i$, with representations $\rho_i^{\text{source}}$ and $\rho_i^{\text{target}}$ on the source and target of the prototype. Namely, 
\[
G_i=MC\{(g_1,g_2)\in\text{Diff}(\C^{s_i},0)\times\text{Diff}(\C^{s_i+\ell},0)\ |\  g_2\circ p_i\circ g_1^{-1}=p_i\}
\]
and the natural representations induced. 
\item $\psi_i: \mathbb Q[[\underline c,\underline s]]\to H^{**}(B\bb T_i)$, the homomorphism induced by sending the abstract variables $c_j$ to the quotient Chern classes of the prototype:
\[
1+c_1+c_2+\dots\mapsto \frac{c(\rho_i^{\text{target}})}{c(\rho_i^{\text{source}})}
\]
and 
\[
s_\lambda\mapsto \psi_i\left(\prod_ic_{\lambda_i}\right).
\]
\end{enumerate}

For a power series $q$, denote the sum of terms of degree lesser or equal to $d$ by $q|_{\leq d}$, and that of degree exactly $d$ by $q|_{\deg d}$. We will denote SSM-Thom polynomials as ssm$_{\mathcal K}(Q,\ell)$ (potentially omitting the $\mathcal K$-equivariance notation if it is understood). This notation is technically redundant, as the data of a local algebra, in the context of determining a singularity, should be thought of as carrying a choice of $\ell$, but we prefer to include it explicitly. 

Given a Mather multisingularity $\underline\eta$, and polynomial of degree $k$, $B_{\underline\eta}\in \mathbb Q[\underline c, \underline s]$, we may ask if the equality

\begin{equation}\label{star}\tag{$\star$}
    B_{\underline\eta}(F)=\ssm\big(\underline{\eta}(F)\big)\big|_{\leq k}
\end{equation}
holds for the $\underline{\eta}$-locus of a $\mathcal C_\ell$-map, $F$, where $B_{\underline \eta}(F)$ indicates that the abstract variables are evaluated at the characteristic classes of $F$ in the usual way. When $F$ is the prototype of a Mather multisingularity, $F=p_i$, then $B_{\underline\eta}(F)$ can be evaluted as $\psi_i(B_{\underline\eta})$. Sending $B$ through $\psi$ is what \cite{rimanyi} refers to as ``considering $f$ $\bb T$-equivariantly'' or ``in $\bb T$-equivariant cohomology'', and this notation is used in the earlier paper of the same year, \cite{rimanyi2025interpolationcharacterizationhigherthom}.  

The following is an interpolation-type theorem from \cite{rimanyi}, Proposition (8.7): We present a version for source Mather multisingularities, and for our particular degree bound (setting their $t$ variable to be $k$). Most notably, within this degree bound, the class of functions $\mathcal C^{Ma,k}_\ell$ introduced there reduces to $\mathcal C_\ell$, as mentioned previously. 

\begin{proposition}\label{thm-interpolation}
    Let $\underline{\eta}$ be a Mather multisingularity, and $B_{\underline{\eta}}\in \bb Q[\underline c,\underline s]$ be a polynomial of degree $k$. TFAE:
    \begin{enumerate}
        \item The polynomial $B_{\underline \eta}$ satisfies $(\star)$ for any $F \in \mathcal C_\ell$.
        \item The polynomial $B_{\underline \eta}$ satisfies $(\star)$ for the prototypes of all nonempty multisingularities $\underline{\xi}$, such that
codim($\underline{\xi}) \leq k$.
    \end{enumerate}
\end{proposition}

In fact, though it was not stated this way in the original paper, Condition 1 can be relaxed to say that the polynomial $B_{\underline\eta}$ satisfies $(\star$) for any $F$ which, up to codimension $\ell$ has only Mather multisingularities, and beyond may have singularities of any type, while $\mathcal C_\ell$-maps have only Mather multisingularities in all codimensions. 

\section{SSM-Thom polynomials via interpolation}\label{main}

The first local algebras, ordered by codimension, are $A_0,A_0^2$, and $A_1$, and within a certain degree bound, these are the only multisingularities that a $\mathcal C_\ell$-map can possess.

\begin{proposition}\label{local algs}
For any $\ell\geq 1$ and $0\leq N<\ell$, 
\[
\text{ssm}(A_0,\ell)\big|_{\leq \ell+N} + \text{ssm}(A_0^2,\ell)\big|_{\leq \ell+N}+\text{ssm}(A_1,\ell)\big|_{\leq\ell+ N}=1.
\]
\end{proposition}

\begin{proof}
Let $F:X \to Y$ be a $\mathcal C_\ell$-map of compact, complex manifolds. The local algebras $A_0,A_0^2$ and $A_1$ are of codimensions 0, $\ell$, and $\ell+1$, respectively. The next lowest codimension is that of $A_0^3$, of codimension $2\ell$. Considering our restriction on $N$, $A_0^3(F)$ must be empty and $F$, being target-locally stable, cannot have singularities of any other type. Therefore

\[
A_0(F)\cup A_0^2(F)\cup A_1(F) = X
\]
The claimed equality then holds by the additivity of SSM classes. 
\end{proof}

This proposition says that determining any two of the SSM-Thom polynomials in this degree range determines the third trivially, and that in this range, there are no other singularities. This proposition is the reason the degree bound of $\ell\geq 1$ and $0\leq N < \ell$ is considered consistently throughout this paper. 

Let us begin by investigating the SSM-Thom polynomial associated to $A_0^2$. The corresponding Thom polynomial is the \textit{double point formula}, which expressed in our (universal) language, reads
\[
\ssm\big(A_0^2,\ell=1\big) = s_\emptyset-c_1\ +\ h.o.t.
\]

Another natural avenue to generalize the double point formula from the perspective of SSM-Thom polynomials is the study of \textit{multiple point formulas}, corresponding to the SSM-Thom polynomial of $A_0^n$ with the same stipulations as above, see \cite{rim-multi-pt}.

For the statement of our results, we introduce some notation: for $m \in \mathbb N$, denote $S_{(m)}$ as the matrix determinant,
\[
S_{(m)} :=\begin{vmatrix}c_1 & c_2 & \dots&c_{m}\\ 1 & c_1 & \dots&c_{m-1}\\0&1&\dots\\ & \vdots & \\
&&1&c_1\end{vmatrix},\quad \text{ and }  \quad S_{(m)}^{LN} := S_{(m)}\Big|_{c\leftrightarrow s}
\]
where $c\leftrightarrow s$ indicates the swapping of $c$ variables in $S_{(m)}$ with $s$ variables by the rule defined on $c$ monomials $\prod_ic_{\lambda_i} \leftrightarrow s_\lambda$ and extending by linearity. For example, when $m=3$ (note that the substitution does not commute with taking the matrix determinant)
\[
S_{(m)}= c_1^3-2c_1c_2+c_3, \quad \quad S^{LN}_{(m)} = s_{111}-s_{12}+s_3.
\]
For the calculations involved in the main theorem, we will need a lemma characterizing the Landweber-Novikov classes for the prototypes we will encounter. These formulas and calculations are well known to experts, but we recall them here for completeness. 

\begin{lemma}\label{LN classes}
Let $f: M \to N$ be a prototype germ of a contact monosingularity, and assume $e(TM)\ne 0$. Then we have
\begin{equation}\label{LN-classes-germ}
s_\emptyset = \frac{f^*e(TN)}{e(TM)}
\end{equation}
where $e$ is the equivariant Euler class. In particular,

\begin{equation}\label{LN-classes-germ-corr}
s_\lambda = c_\lambda \cdot s_\emptyset,\quad\quad S_{(N)}^{LN}=S_{(N)}\cdot s_\emptyset.
\end{equation}
If $f$ is a prototype of a contact multisingularity of the form\footnote{This lemma can easily be generalized, but this is the only form we will need for our present applications.}, 

\[
f: M^{(1)}\sqcup M^{(2)}\to N
\]
where both $e(TM^{(1)})$ and $e(TM^{(2)})$ are non-zero, we have

\begin{equation}\label{LN-classes-multigerm}
s_\lambda = c_\lambda^{(1)}\frac{f^*e(TN)}{e(TM^{(1)})} + c_\lambda^{(2)} \frac{f^*e(TN)}{e(TM^{(2)})}
\end{equation}
where $c_i=(c_i^{(1)},c_i^{(2)})$. Further,

\begin{equation}\label{LN-classes-multigerm-corr}
S^{LN}_{(N)} = S_{(N)}\left(c^{(1)}_1,\dots,c_N^{(1)}\right)\frac{f^*e(TN)}{e(TM^{(1)})} + S_{(N)}\left(c^{(2)}_1,\dots,c^{(2)}_N\right)\frac{f^*e(TN)}{e(TM^{(2)})}.
\end{equation}
\end{lemma}

\begin{proof}
In the first case, suppose $f:M\to N$ is a prototype function of a contact singularity. By definition, $s_\emptyset=f^*f_*(1)$. Let $\alpha\in H^*(N)$ be a class satisfying $f^*\alpha=e(TM)$. This can always be done because for germs, pullback is an isomorphism on cohomology.

Recall the well-known (cohomological) projection formula,

\begin{lemma}
If $f$ is a map of compact, orientable manifolds, then for $\alpha \in H^*(N)$ and $\beta \in H^*(M)$,

\[
f_*(f^*\alpha\smile \beta)=\alpha\smile f_*(\beta)
\]
in which push-forward in cohomology, the ``wrong-way map'', is defined using Poincar\'e duality on both manifolds, and homological push-forward.

\end{lemma}

Setting $\beta=1$ and suppressing cup product notation, we obtain the ``push-pull'' formula,

\[
f_*f^*\alpha=\alpha f_*(1).
\]
For example, when $\dim(M)=\dim(N)$, $f_*(1)$ is the familiar notion of degree. Then for our choice of $\alpha$,
\begin{align*}
f_*f^*\alpha&=\alpha f_*(1)\\
f_*\big(e(TM)\big)&=\alpha f_*(1)\\
f^*f_*\big(e(TM)\big) &=f^*\alpha f^*f_*(1).
\end{align*}
Of course, $f_*(e(TM))=e(TN)$, and applying the projection formula again on the RHS, 
\[
f^*\big(e(TN)\big)=e(TM)f^*f_*(1).
\]
Because $e(TM)\ne 0$, we obtain
\[
s_\emptyset\equiv f^*f_*(1)=\frac{f^*\big(e(TN)\big)}{e(TM)}.
\]
The first formula of (\ref{LN-classes-germ-corr}) follows from another application of the projection formula, and the second follows from the first.

The formulas (\ref{LN-classes-multigerm}) and (\ref{LN-classes-multigerm-corr}) follow by the same calculations as above, where the cohomological unit is a sum of the units for $M^{(1)}$ and $M^{(2)}$. 

\end{proof}

\begin{remark}
    The equivariant Euler class may be computed as a product over the weights of the symmetry group action. The assumption that $e(TM)\ne 0$ is equivalent to saying that no tangent direction has equivariant weight 0. It is a happy coincidence that this criteria is satisfied for all prototypes of multisingularities, which can be verified experimentally by checking all of them.
    
\end{remark}

\subsection{SSM-Thom polynomial of $A_0^2$}

\begin{theorem}\label{SSM of A02}
Let $\ell\geq1$. For all $0\leq N<\ell$, the degree $\ell+N$ component of ssm$(A_0^2,\ell)$ is

\begin{equation} \label{eq1}
\ssm(A_0^2,\ell)\Big|_{\text{deg }\ell+N}=
(-1)^{N+1}\sum_{k=0}^NS_{(N-k)}{\ell+N-1\choose k}c_{\ell+k}+(-1)^{N} S^{LN}_{(N)}.
\end{equation}

\end{theorem}

Example (\ref{ex: SSM A02}) provides a visual display of the degree ranges being considered. 

\begin{proof}
The chosen degree bound satisfies $\ell+N\leq M(\ell)$ for all $\ell\geq 1$. As in Proposition (\ref{local algs}), the local algebras of contact singularities of codimension $\leq \ell+N$ are only $Q_1=A_0,Q_2=A_0^2$, and $Q_3=A_1$, of codimensions 0, $\ell$ and $\ell+1$, respectively. According to the interpolation theorem Proposition (\ref{thm-interpolation}), we then need to verify Condition (\ref{star}) in 3 cases: 

\textbf{Case $Q_1=A_0$:} The prototype is
\[
p_1: \{\text{pt}\}\to \C^{\ell}
\]
with maximal torus $U(1)^\ell$. The source representation is trivial, and the target representation is given by the defining representation
\[
\rho^{\text{target}}_1(\alpha_1,\dots,\alpha_\ell)\cdot (z_1,\dots,z_\ell)=(\alpha_1z_1,\dots,\alpha_\ell z_\ell).
\]
Then the map $\psi_1$ assigns via the splitting principle
\[
c = \prod_{i=1}^{\ell}(1+\alpha_i)\quad \Rightarrow \quad c_i=e_i(\alpha_1,\dots,\alpha_\ell)
\]
where $e_i(X_1,\dots,X_n)$ is the elementary symmetric polynomial in $n$ variables, and
\[
s_\emptyset = e(T\C^\ell)=\alpha_1\cdots \alpha_\ell,\quad \quad s_\lambda = c_\lambda \cdot s_\emptyset.
\]
In particular, $c_k=0$ for any $k>\ell$, and $s_\emptyset=c_\ell$. Then within the summation in (\ref{eq1}), $\ssm(A_0^2,\ell)|_{\text{deg }\ell+N}$, has only one non-zero term corresponding to $k=0$.
Also, 
\[
S_{(N)}^{LN} = S_{(N)}\cdot s_\emptyset.
\]
The LHS of (\ref{star}) is thus 
\[
(-1)^{N+1}\begin{vmatrix}e_1 & e_2 & \dots&e_{N}\\ 1 & e_1 & \dots&e_{N-1}\\0&1&\dots\\ & \vdots & \\
&&1&e_1\end{vmatrix}e_\ell+(-1)^{N}S_{(N)}^{LN}=0
\]
while the RHS is the SSM class of the double point locus of $p_1$. This locus is the empty set, thus as well, the RHS of (\ref{star}) is 0. 

\textbf{Case $Q_2=A_0^2$:} The prototype is 
\[
p_2: \C_1^\ell\sqcup \C_2^\ell\to\C^{2\ell}
\]
\[
z=(z_1,\dots,z_\ell)\mapsto \begin{cases}
(z_1,\dots,z_\ell,\underbrace{0,\dots,0}_\ell) & z\in \C^\ell_1\\
(\underbrace{0,\dots,0}_\ell,z_1,\dots,z_\ell) & z \in C^\ell_2
\end{cases}
\]
with maximal torus $U(1)^{2\ell}$ and a pair of source and target representations, $(\rho_2^{\text{source}},\rho_2^{\text{target}})$ and $((\rho_2^{\text{source}})',\rho_2^{\text{target}})$, where
\[
\rho_2^{\text{target}}(\alpha_1,\dots,\alpha_\ell,\beta_1,\dots,\beta_\ell)\cdot (z_1,\dots,z_{2\ell}) =\big(\alpha_1 z_1,\dots,\beta_{\ell}z_{2\ell}\big)
\]

\[
\rho_2^{\text{source}}(\alpha_1,\dots,\alpha_\ell,\beta_1,\dots,\beta_\ell)\cdot (z_1,\dots,z_\ell) = \begin{cases}
(\alpha_1 z_1,\dots,\alpha_\ell z_\ell) & z\in \C^\ell_1\\
0 & z \in \C^\ell_2
\end{cases}
\]

\[
(\rho_2^{\text{source}})'(\alpha_1,\dots,\alpha_\ell,\beta_1,\dots,\beta_\ell)\cdot (z_1,\dots,z_\ell) = \begin{cases}
0 & z \in \C^\ell_1\\
(\beta_1 z_1,\dots,\beta_\ell z_\ell) & z\in \C^\ell_2
\end{cases}.
\]

Due to the symmetry in the source, it suffices to consider only $(\rho_2^{\text{source}},\rho_2^{\text{target}})$, as Condition (\ref{star}) for the two representations will be identical. Interestingly, as in Lemma (\ref{LN classes}), the second pair of source and target representations still contribute through the Landweber-Novikov classes: This is one facet particular to the case of multigerms.

Next, $\psi_2$ assigns

\[
c^{(1)}=\frac{\prod_{i=1}^{\ell}(1+\alpha_i)(1+\beta_i)}{\prod_{i=1}^{\ell}(1+\alpha_i)}\quad \Rightarrow \quad c_i^{(1)} = e_i(\beta_1,\dots,\beta_\ell)
\]

\[
c^{(2)}=\frac{\prod_{i=1}^{\ell}(1+\alpha_i)(1+\beta_i)}{\prod_{i=1}^{\ell}(1+\beta_i)}\quad \Rightarrow \quad c_i^{(2)} = e_i(\alpha_1,\dots,\alpha_\ell)
\]

and, because $p_2$ is a multigerm,

\[
S^{LN}_{(N)}=S_{(N)}\left(c^{(1)}_1,\dots,c_N^{(1)}\right)\prod_{i=1}^\ell\beta_i + S_{(N)}\left(c^{(2)}_1,\dots,c^{(2)}_N\right)\prod_{i=1}^\ell\alpha_i.
\]
In particular, $c_k=0$ for any $k>\ell$ again. Then the LHS of (\ref{star}) is 
\[
(-1)^{N+1}S_{(N)}\Big(c_1^{(1)},\dots,c_N^{(1)}\Big) \cdot\prod_{i=1}^{\ell}\beta_i+ (-1)^{N}\Big( S_{(N)}(c_1^{(1)},\dots,c_{(N)}^{(1)})\cdot\prod_{i=1}^{\ell}\beta_i+S_{(N)}\left(c^{(2)}_1,\dots,c^{(2)}_N\right)\prod_{i=1}^\ell\alpha_i\Big)
\]
\[
=(-1)^N\cdot S_{(N)}\left(c^{(2)}_1,\dots,c^{(2)}_N\right)\prod_{i=1}^\ell\alpha_i.
\]
As $c_i^{(2)}=e_i(\alpha_1,\dots,\alpha_\ell)$, (for readability, we shorten the notation to $e_i(\alpha)$) we have 
\[
=(-1)^N \begin{vmatrix}e_1(\alpha) & e_2(\alpha) & \dots&e_N(\alpha)\\ 1 & e_1(\alpha) & \dots&e_{N-1}(\alpha)\\0&1&\dots\\ & \vdots & \\
&&1&e_1(\alpha)\end{vmatrix}\prod_{i=1}^\ell \alpha_i.
\]
This matrix determinant is exactly the expansion of $h_N(\alpha_1,\dots,\alpha_\ell)$, the complete, homogeneous, symmetric polynomial in $\alpha_1,\dots,\alpha_\ell$, into elementary symmetric polynomials, which generate the algebra of symmetric functions:
\[
=(-1)^Nh_N(\alpha_1,\dots,\alpha_\ell)\prod_{i=1}^\ell\alpha_i.
\]
Meanwhile the RHS of (\ref{star}) is the SSM class of the double point locus of $p_2$. The locus of double points in $\C^{(1)}\sqcup \C^{(2)}$ is the pair of origins, thus the component in $\C^\ell_1$ consists of only the origin:
\[
\text{RHS} = \ssm\Big(A_0^2(p_2),\ell\Big)\Big|_{\deg \ell+N}=\ssm\big(0\subset \C^\ell_1,\ell\big)\Big|_{\deg \ell+N}=\frac{\csm\big(0\subset\C_1^\ell,\ell\big)}{c(T\C_1^\ell)}\Bigg|_{\deg \ell+N}.
\]
The inclusion $\iota:\{0\}\hookrightarrow \C^\ell$ is smooth, so the normalization axiom implies

\[
\csm\big(0\subset\C_1^\ell,\ell\big) = \iota_*(1)
\]

\[
=e(T\C^\ell_1)=\prod_{i=1}^{\ell}\alpha_i
\]
\[
\Rightarrow \text{RHS} =\frac{\prod_{i=1}^\ell\alpha_i}{\prod_{i=1}^\ell(1+\alpha_i)}\Bigg|_{\text{deg }\ell+N}=\prod_{i=1}^\ell\alpha_i \cdot h_N(-\alpha_1,\dots,-\alpha_\ell) = (-1)^N\prod_{i=1}^\ell\alpha_i\cdot h_N(\alpha_1,\dots,\alpha_\ell)
\]
as required. Finally, we check

\textbf{Case $Q_3 = A_1$:} The prototype is

\[
p_3 : \C^{\ell+1}\to \C^{2\ell+1}
\]
\[
(x,y_1,\dots,y_\ell)\mapsto (x^2,xy_1,\dots,xy_\ell,y_1,\dots,y_\ell)
\]
with maximal torus $U(1)^{\ell+1}$ and representations 

\[
\rho^{\text{source}}(\alpha,\beta_1,\dots,\beta_\ell)\cdot(x,y_1,\dots,y_\ell) = (\alpha x,\beta_1 y_1,\dots,\beta_\ell y_\ell)
\]
\[
\rho^{\text{target}}(\alpha,\beta_1,\dots,\beta_\ell)\cdot(x,y_1,\dots,y_\ell,z_1,\dots,z_\ell) = \big(\alpha^2 x,\alpha\beta_1 y_1,\dots,\alpha\beta_\ell y_\ell,\beta_1 z_1,\dots,\beta_\ell z_\ell\big).
\]
Computation of the LHS will be difficult, so let us begin with the RHS, which features a nice non-trivial example in which calculating the multisingularity locus of a prototype, and its subsequent CSM class, is actually possible, the latter using mostly only the defining axioms of the CSM class. The class of such examples is very small, and in general is a hopeless task.

The $A_0^2$ locus of $p_3$ is such that two distinct points of the source go to one in the target: Let $p_3(x,y_1,\dots,y_\ell) = p_3(x', y'_1,\dots,y'_\ell)$. By definition of $p_3$, 

\[
y_i=y'_i, \quad i=1,\dots,\ell
\]
\[
xy_i=x'y'_i, \quad i=1,\dots,\ell
\]
\[
x^2=(x')^2.
\]
If there is at least one $y_i\ne 0$, then the second equation implies $x=x'$, so the points are no longer distinct. Therefore the locus may only consist of points of the form $(x,0,\dots,0)$, under which $(x,0,\dots,0)$ and $(-x,0,\dots,0)$ are two to one, as long as $x\ne 0$. Thus, the locus of smooth double points is the $x$-axis in $\C^{\ell+1}$, denoted $\C_x$, minus the origin: 

\[
A_0^2(p_3) = \C_x- 0.
\]
By additivity, the CSM class is computed as
\begin{align*}
\csm\big(\C_x-0\big) &= \csm(\C_x)-\csm(0)\\
&=\iota_*c(T\C_x) - \iota_*(1)\\
&=c(T\C_x)\iota_*(1)-\iota_*(1)\\
&=(1+\alpha)\beta_1\cdots\beta_\ell - \alpha\beta_1\cdots\beta_\ell\\
&=\beta_1\cdots\beta_\ell.
\end{align*}
So we obtain
\[
\text{RHS (\ref{star})} = \ssm\big(A_0^2(p_3),\ell\big)\Big|_{\ell+N} = \frac{\beta_1\cdots\beta_\ell}{(1+\alpha)(1+\beta_1)\cdots(1+\beta_\ell)}\Bigg|_{\ell+N}.
\]
Taking the degree $\ell+N$ component,
\[
\text{RHS}=(-1)^N\prod_{i=1}^\ell\beta_i \cdot h_N(\alpha,\beta_1,\dots,\beta_\ell).
\]
For the LHS of (\ref{star}), $\psi_3$ assigns\footnote{The difficulty in analyzing the case $Q=A_1$ comes from exactly this fraction: There is no simple expression for the degree $k$ component of its Taylor expansion.} 

\[
c(p_3) = \frac{(1+2a)(1+a+b_1)\cdots (1+a+b_\ell)(1+b_1)\cdots(1+b_\ell)}{(1+a)(1+b_1)\cdots (1+b_\ell)}.
\]
Expanding Taylor series individually, we can write 
\[
c_k=e_k(\alpha+\beta_1,\dots,\alpha+\beta_\ell) + (-1)^{k+1}\alpha^k+\sum_{a=1}^{k-1 } e_a(\alpha+\beta_1,\dots,\alpha+\beta_\ell)\cdot (-1)^{-(k-a)+1}\alpha^{k-a} 
\]
where $e(-)$ is again the elementary symmetric polynomial. When $k>\ell$, $e_k=0$, so that
\[
c_{k>\ell} = \sum_{a=0}^\ell e_a(\alpha+\beta_1,\dots,\alpha+\beta_\ell)\cdot(-1)^{-(k-a)}\alpha^{k-a}
\]
\[
=\alpha^{k-\ell}(-1)^{k-\ell-1}\left(\sum_{a=0}^\ell e_a\cdot (-1)^{\ell-a}\alpha^{\ell-a}\right).
\]
Applying the elementary symmetric polynomial identity
\[
\prod_{j=1}^n(\lambda-X_j) = \lambda^n-e_1(X_1,\dots,X_n)\lambda^{n-1} + e_2(X_1,\dots,X_n)\lambda^{n-2}+\dots+(-1)^ne_n(X_1,\dots,X_n)
\]
we obtain, for $k>0$, 
\[
c_{\ell+k} = (-1)^{k+1}\alpha^k\beta_1\cdots \beta_\ell.
\]
This differs from the previous two cases, as now all $c_k$ are non-zero. Because $p_3$ is a monogerm, we have 
\[
s_\emptyset = 2\prod_{i=1}^\ell(\alpha+\beta_i) \equiv 2e_\ell(\alpha+\beta_1,\dots,\alpha+\beta_\ell)
\]
\[
S_{(N)}^{LN} = S_{(N)}\cdot s_\emptyset.
\]
Therefore the LHS of (\ref{star}) is 
\[
(-1)^{N+1}\sum_{k=0}^N S_{(N-k)}{\ell+N-1\choose k}c_{\ell+k} + (-1)^NS^{LN}_{(N)}
\]
\[
=(-1)^{N+1}S_{(N)}(c_\ell-s_\emptyset) + (-1)^{N+1}\sum_{k=1}^N S_{(N-K)}{\ell+N-1\choose k}c_{\ell+k}.
\]
Applying the previous elementary symmetric polynomial identity again, we obtain 
\[
c_\ell-s_\emptyset = -\beta_1\cdots \beta_\ell.
\]
Combining this with the formula for $c_{\ell+k}$, LHS can be simplified to 
\[
\text{LHS} = (-1)^{N+1}\left[\sum_{k=0}^NS_{(N-k)}{\ell+N-1\choose k}(-1)^{k+1}\alpha^k\right]\prod_{i=1}^{\ell}\beta_i.
\]
Therefore to complete this step, the final step, it suffices to show 
\[
h_N\big(\alpha,\beta_1,\dots,\beta_\ell\big) = \left[\sum_{k=0}^NS_{(N-k)}{\ell+N-1\choose k}(-\alpha)^k\right].
\]
Let us relegate the proof of this (purely algebraic) statement to the appendix, see Lemma (\ref{last-step}), though its proof does involve some interesting algebraic/combinatorial techniques. With the final prototype checked, the proof is complete. 
\end{proof}

\subsection{SSM-Thom polynomials of $A_0$ and $A_1$}\label{A0, A1 subsection}

In a similar fashion, we can prove

\begin{theorem}\label{ssm A1}
In the same setting as Theorem (\ref{SSM of A02}),
\[
\text{ssm}(A_1,\ell)\Big    |_{\text{deg }\ell+N} = (-1)^{N-1}\sum_{k=0}^{N-1}S_{(N-1-k)}(c_1,\dots,c_{N-1-k})\cdot {\ell+N-2\choose k}\cdot c_{\ell+k+1}
\]
\end{theorem}
\begin{proof}
    We must check the same three (\ref{star}) conditions, for $A_0,A_0^2$ and $A_1$, and conveniently, the LHS of Condition (\ref{star}) has already been spelled out explicitly in the proof of Theorem (\ref{SSM of A02}) for a nearly identical formula, so we explain the checks briefly, using the same notations.
    
    \textbf{Case $A_0$}: For the $A_0$ prototype, $p_1$, Condition (\ref{star}) is satisfied automatically as both sides are 0: The $A_1$ locus of the prototype $p_1$ is empty, so the RHS of (\ref{star}) is 0, and so is the LHS, since $c_i(p_1)=e_i(\alpha_1,\dots,\alpha_\ell)$, so $c_k=0$ for $k>\ell$, and all Chern classes appearing in our formula contain a factor of degree $>\ell$, which is not the case for the formula for ssm($A_0^2,\ell)$, which has an initial term beginning with a factor of $c_\ell$. 

    \textbf{Case $A_0^2$}: Follows the same as the previous case, in contrast to the proof of Theorem (\ref{SSM of A02}), in which the $A_0^2$ check was non-trivial. This is for 2 reasons: 1) The formula we are checking currently features no Landweber-Novikov classes, and 2) The $A_1$ locus of $p_2$ is still empty, while the $A_0^2$ locus is not. Again both sides of Condition (\ref{star}) are 0, for the same reasons as previous case. 

    \textbf{Case $A_1$:} Finally, the $A_1$ locus of the $A_1$ prototype is the origin in $\C^\ell$. So the RHS of Condition (\ref{star}) is 
    \begin{align*}
    \text{RHS (\ref{star}}) = \ssm\big(0\subset \C^\ell\big)\big|_{\deg\ell+N} &= \frac{\csm\left(0\subset \C^\ell\right)}{(1+\alpha)(1+\beta_1)\cdots(1+\beta_\ell)}\Bigg|_{\deg\ell+N}.
    \end{align*}
By normalization, $\csm(0\subset \C^\ell)=i_*(1)$, where $i: 0\to \C^\ell$, so
    \begin{align*}
    &=\frac{\iota_*(1)}{(1+\alpha)(1+\beta_1)\cdots(1+\beta_\ell)}\Bigg|_{\deg\ell+N}\\
    &=\frac{\alpha\beta_1\cdots\beta_\ell}{(1+\alpha)(1+\beta_1)\cdots(1+\beta_\ell)}\Bigg|_{\deg\ell+N}\\
    &=(-1)^{N-1}h_{N-1}(\alpha,\beta_1,\dots,\beta_\ell)\cdot \alpha\beta_1\cdots \beta_\ell.
    \end{align*}
While the left hand side is 
    \begin{align*}
        \text{LHS (\ref{star})} = (-1)^{N-1}\sum_{k=0}^{N-1}S_{(N-1-k)}{\ell+N-2\choose k}c_{\ell+k+1}.
    \end{align*}
By the formula for $c(p_3)$ in degree strictly greater than $\ell$, 
    \begin{align*}
        &=(-1)^{N-1}\sum_{k=0}^{N-1}S_{(N-1-k)}{\ell+N-2\choose k}(-1)^k\alpha^{k+1}\beta_1\cdots\beta_\ell\\
        &=(-1)^{N-1}\sum_{k=0}^{N-1}S_{(N-1-k)}{\ell+N-2\choose k}(-\alpha)^k\cdot \alpha\beta_1\cdots \beta_\ell. 
    \end{align*}
Now verifying that LHS=RHS is exactly Lemma (\ref{last-step}) from the appendix, with $N$ replaced by $N-1$, the same final step as in the previous theorem. 
\end{proof}

As mentioned previously, determining two of the three SSM-Thom polynomials automatically yields the third.

\begin{corollary}\label{ssm A0}
In the same setting as Theorem (\ref{SSM of A02}),
\[
\text{ssm}(A_0,\ell)\Big|_{\text{deg }\ell+N} = \sum_{k=0}^NS_{(N-k)}\cdot c_{\ell+k}\cdot \left[{\ell+N-1\choose k}+{\ell+N-2\choose k-1}\right] 
 + (-1)^{N+1}S^{LN}_{(N)}
 \]
\end{corollary}

\begin{proof}
Follows from Theorems (\ref{SSM of A02}), (\ref{ssm A1}), and Proposition (\ref{local algs})

\end{proof}

\section{Geometry of singularity loci of maps of projective spaces}

We show one application of calculating higher order extensions of the double point formula. In particular, once a formula for the SSM-Thom polynomial is determined, the Thom principle gives a concrete formula for the SSM class of the $\eta$-locus of \textit{any} $\mathcal C_\ell$-map. A theorem of Aluffi-Ohmoto connects the SSM class of a variety to the geometry of the variety, which we now recall. 

Let $X$ be a subvariety of $\bb P^n$, and define 
\[
X_r:=X \cap H_1\cap \dots \cap H_r
\]
where $H_i$ are general hyperplanes (so that $X_0=X$). We can consider the topological Euler characteristics of $X$ and $X_r$, and assemble these into a generating polynomial,
\[
\chi_X(t):=\sum_{r\geq 0}\chi(X_r)\cdot (-t)^r.
\]
$X_r$ will eventually be empty for high enough $r$, so this really is a polynomial. Let $\gamma_X(t) = \sum_{r\geq 0}\gamma_r t^r$ be the polynomial obtained from $\csm(X\subset \bb P^r)$ by replacing $[\bb P^r]$ with $t^r$. Define a transformation on the set of polynomials by the formula
\[
p(t)\mapsto \mathscr I(p):=\frac{tp(-t-1)+p(0)}{t+1}.
\]
$\mathscr I$ is an involution. 
\begin{theorem}[\cite{aluffi}\cite{ohm-integral}]\label{linear sections}
For every locally closed set $X\subset \mathbb P^n$, the involution $\mathscr I$ interchanges $\chi_X(t)$ and $\gamma_X(t)$. 
\end{theorem}

In other words, if you know the SSM class of a subvariety of $\mathbb P^n$, you can obtain the Euler characteristics of general linear sections of $X$ through some simple polynomial transformations. 

For any $\ell\geq 2$, let $F: \mathbb P^{\ell+1}\to \mathbb P^{2\ell+1}$ be a $\mathcal C_\ell$-map of degree $d$, and let $\Sigma :=A_0^2(F)\subset \mathbb{P}^{\ell+1}$. This family of maps is chosen so that only the first order deformations of the Thom polynomials appear, for simplicity. In general for a chosen domain dimension, $N$, the higher order deformations we calculated in Section 3 can be utilized to obtain (non-trivial) analogous results in dimensions $\bb P^N\to \bb P^{N+\ell}$ for $\ell \in \left(\frac n2,n\right)$, non-inclusive. The family considered in this section corresponds to the first interesting case, $\ell=\lceil\frac n2\rceil$.

Target-local stability of $F$ for this choice of domain and codomain dimension implies that $F$ can have singularities only of types $A_0,A_0^2$ and $A_1$, and they must be of dimensions $\ell+1,1$, and 0, respectively. The closure $A_0^2(F)$ is then 
\[
\overline{A_0^2(F)} = A_0^2(F) \cup A_1(F).
\]

\begin{proposition}
We have that $A_1(F) \subset A_0^2(F)$ is smooth.
\end{proposition}

\begin{proof}
Suppose $p \in A_0^2$ is also an $A_1$ point. Because the $A_1$ locus is of dimension 0, we can assume $p$ is isolated, and thus has a neighborhood (a 1-dimensional neighborhood  within the $A_0^2$ locus) of $A_0^2$ points. As remarked in the proof of Theorem (\ref{SSM of A02}), the $A_0^2$ locus of the prototype for $A_1$ is of the form $(*,0,\dots,0)\subset \C^{\ell+1}$, which is a smooth inclusion. 

\end{proof}

In particular, $\overline{A_0^2(F)}$ is a Riemann sphere with marked points corresponding to the $A_1$ locus. 

To apply Theorem (\ref{linear sections}) we must compute the CSM class, and ultimately we will show that the topology of the locus is determined only by $d$, so we want to express the SSM classes in terms of $d$. From the theorems of Section 4, we have 

\begin{gather*}
\operatorname{ssm}(A_0(F),\ell) = 1 + (c_\ell-s_0) +(-c_1c_\ell-(\ell+1)c_{\ell+1}+s_1)\\
= 1 + (c_\ell-s_0) +(-c_1(c_\ell-s_\emptyset)-(\ell+1)c_{\ell+1})\\\\
\operatorname{ssm}(A_0^2(F),\ell) = (s_0-c_\ell) + (c_1c_\ell + \ell c_{\ell+1}-s_1) 
\\
=(s_\emptyset-c_\ell)+(c_1(c_\ell-s_\emptyset)+\ell c_{\ell+1})\\\\
\operatorname{ssm}(A_1(F),\ell) = c_{\ell+1} 
\end{gather*}
where $c$ and $s$ are the quotient Chern and Landweber-Novikov classes of $F$, and all higher order terms are 0. Let $h$ and $H$ denote the hyperplane class of $\mathbb P^{\ell+1}$ and $\bb P^{2\ell+1}$, respectively. As $F$ is of degree $d$,
\[
F^*(H)=dh.
\]
By definition, 
\[
c(F) = \frac{F^*(1+H)^{2\ell+2}}{(1+h)^{\ell+2}}\equiv \frac{(1+dh)^{2\ell+2}}{(1+h)^{\ell+2}}.
\]
The degree $k$ portion can be written as (now we suppress the $F$ dependence notation)

\[
c_k = (-h)^k\sum_{a=0}^k{2\ell+2\choose a}{k-a+\ell+1\choose \ell+1}(-d)^a.
\]
The only terms involved in this example are $c_1, c_\ell$ and $c_{\ell+1}$, which are thus 

\begin{gather*}
c_1=(-\ell-2+2\ell d+2d)h\\
c_\ell = h^\ell\sum_{a=0}^\ell{2\ell+2\choose a}{2\ell-a+1\choose \ell+1}d^a(-1)^{\ell-a}\\
c_{\ell+1} = {2\ell+2\choose \ell+1}(d-1)^{\ell+1}h^{\ell+1}.
\end{gather*}
The Landweber-Novikov classes can be computed as
\begin{lemma}
\[
s_\emptyset = d^{2\ell+1}h^\ell
\]
\[
s_1=d^{2\ell+1}\big(-\ell-2+2\ell d+2d\big)h^{\ell+1}.
\]
\end{lemma}
\begin{proof}
The first equation can be deduced from exactly the same computations as in Lemma (\ref{LN classes}), with only one modification. A crucial step in the proof of Lemma (\ref{LN classes}) is the assumption that every cohomology class in the source is the pullback of a class from the target, which held because $F$ was a germ. Now $F$ is not a germ but a global map, but in order to apply the projection formula, we still need this assumption. In this case, it is still true because we study maps of projective spaces: every cohomology class in the source is (rationally) a pullback of a class in the target. 

The second equation follows from the first, also with projection formula and the explicit formula for $c_1(F)$. 
\end{proof}
Using synthetic division, we can calculate that 
\begin{gather*}
 c_\ell-s_\emptyset = -(d-1)\left(-\sum_{\alpha=\ell}^{2\ell}d^\alpha+(-1)^\ell\sum_{n=0}^{\ell-1}\left(\sum_{\alpha=0}^{n}(-1)^\alpha{2\ell+2\choose \alpha}{2\ell+1-\alpha\choose \ell+1}+(-1)^{\ell+1}\right)d^n\right).
\end{gather*}
It will not be especially productive to refer to the quantity $c_\ell-s_\emptyset$ in this form, but it is notable that $(d-1)$ divides it. In what follows, we abbreviate the coefficient of $c_\ell$ as $\Sigma_\ell$,

\[
\Sigma_\ell:=\sum_{a=0}^\ell{2\ell+2\choose a}{2\ell-a+1\choose \ell+1}d^a(-1)^{\ell-a}
\]
thus write $c_\ell-s_\emptyset\equiv(\Sigma_\ell-d^{2\ell+1})h^\ell$. The point is this is a finite, easily computable quantity. Substituting in, we obtain
\begin{gather*}
\text{ssm}(A_0(F),\ell) = 1 + (\Sigma_\ell-d^{2\ell+1})h^\ell + \\
\Big((\ell+2-2\ell d-2d)\cdot (\Sigma_\ell-d^{2\ell+1}) - (\ell+1){2\ell+2\choose \ell+1}(d-1)^{\ell+1}\Big)h^{\ell+1}\\\\
\text{ssm}(A_0^2(F),\ell) =\left(d^{2\ell+1}-\Sigma_\ell\right)h^\ell \\
+ \left((-\ell-2+2\ell d+2d)\cdot (\Sigma^\ell-d^{2\ell+1})+(\ell){2\ell+2\choose \ell+1}(d-1)^{\ell+1} \right)h^{\ell+1}\\\\
\text{ssm}(A_1(F),\ell) = {2\ell+2\choose \ell+1}(d-1)^{\ell+1}h^{\ell+1}.
\end{gather*}
Let us also remark that 

\begin{gather*}
\overline{A_0^2(F)} = A_0^2(F) \sqcup A_1(F),
\end{gather*}
so again using additivity, 

\begin{gather*}
\ssm(\overline{A_0^2(F)},\ell) = \ssm(A_0^2(F),\ell)+\ssm(A_1(F),\ell)\\\\
=\left(d^{2\ell+1}-\Sigma_\ell\right)h^\ell \\
\left((-\ell-2+2\ell d+2d)\cdot (\Sigma_\ell-d^{2\ell+1})+(\ell+1){2\ell+2\choose \ell+1}(d-1)^{\ell+1} \right)h^{\ell+1}.
\end{gather*}
To convert to CSM classes, we multiply by $c(T\mathbb P^{\ell+1})=(1+h)^{\ell+2}\big|_{h^{\ell+2}}=0$, and obtain

\begin{gather*}
\text{csm}(A_0(F),\ell) 
=1+{\ell+2\choose 1}h+{\ell+2\choose 2}h^2 + \dots + {\ell+2\choose \ell-1}h^{\ell-1}\\
+\Big[(\Sigma_\ell-d^{2\ell+1}) + {\ell+2\choose \ell}\Big]h^\ell\\
+\left[(2\ell+4-2\ell d-2d)(\Sigma_\ell-d^{2\ell+1})-(\ell+1){2\ell+2\choose \ell+1}(d-1)^{\ell+1}+\ell+2\right]h^{\ell+1}
\end{gather*}

\begin{gather*}
\csm(A_0^2(F),\ell) =(d^{2\ell+1}-\Sigma_\ell)h^\ell \\
+\left(2(\ell d+d-\ell-2)\cdot (\Sigma_\ell-d^{2\ell+1})+(\ell){2\ell+2\choose \ell+1}(d-1)^{\ell+1} \right)h^{\ell+1}
\end{gather*}
and finally,

\[
\text{csm}(A_1,\ell)= {2\ell+2\choose \ell+1}(d-1)^{\ell+1}h^{\ell+1}.
\]
All formulas are now written depending only $d,\ell$. We assemble the polynomial Aluffi refers to as $\gamma_{Q}(t)$ by interchanging powers of $h$ and $t$ reciprocally, and obtain 

\begin{gather*}
\gamma_{A_0}(t) = t^{\ell+1} + (\ell+2)t^\ell + {\ell+2\choose 2}t^{\ell-1} + \dots + {\ell+2\choose \ell-1}t^2\\
+ \Big(\Sigma_\ell-d^{2\ell+1}+{\ell+2\choose \ell}\Big)t\\
\left[(2\ell+4-2\ell d-2d)(\Sigma_\ell-d^{2\ell+1})-(\ell+1){2\ell+2\choose \ell+1}(d-1)^{\ell+1}+\ell+2\right]\\\\
 \gamma_{A_0^2}(t) = (d^{2\ell+1}-\Sigma_\ell)t + \Big(2(-\ell d+d-\ell-2)\cdot (\Sigma_\ell-d^{2\ell+1})+(\ell){2\ell+2\choose \ell+1}(d-1)^{\ell+1}\Big)\\\\
 \gamma_{A_1}(t)={2\ell+2\choose \ell+1}(d-1)^{\ell+1}.
\end{gather*}
Next we apply involutions $\mathscr I_X(t)$. An equivalent way to formulate the transformation $\mathscr I$ is to separate polynomial $p$ into its degree 0 part and degree $>0$ part, $p(t) = p(0)+t\cdot p_+(t)$. Then 

\[
\mathscr I(p(t)) = p(0) - t\cdot p_+(-t-1).
\]
In this form it is clear to see that $\gamma_{A_1}$ is unchanged, and $\gamma_{A_0^2}$ has only a sign change on its linear term,

\begin{gather*}
\mathscr I(\gamma_{A_0^2})(t) =  (\Sigma_\ell-d^{2\ell+1})t + \Big(2(-\ell d+d-\ell-2)\cdot (\Sigma_\ell-d^{2\ell+1})+\ell{2\ell+2\choose \ell+1}(d-1)^{\ell+1}\Big)\\\\
\mathscr I(\gamma_{A_1})(t) ={2\ell+2\choose \ell+1}(d-1)^{\ell+1}.
\end{gather*}
Computing $\mathscr I(\gamma_{A_0})(t)$ requires more work but yields satisfying answer: In particular, the $t^2$ and above terms are independent of $d$, which we will see later must have been the case,

\begin{lemma}
\begin{gather*}
\mathscr I(\gamma_{A_0})(t) = \left[(2\ell+4-2\ell d-2d)(\Sigma_\ell-d^{2\ell+1})-(\ell+1){2\ell+2\choose \ell+1}(d-1)^{\ell+1}+\ell+2\right]\\
-t\cdot \left[\left(\Sigma_\ell-d^{2\ell+1}+\ell+1\right) + \sum_{n=1}^\ell (-t)^n\cdot(\ell+1-n)\right].
\end{gather*}
\end{lemma}
We also assemble the generating function of topological Euler characteristics of general linear sections:

$$
\chi_X(t) = \chi(X) - \chi(X \cap H) t + \chi(X\cap H_1\cap H_2)t^2 - \dots
$$
and Theorem (\ref{linear sections}) states that
$$
\chi_X(t) = \mathscr I(\gamma_{X})(t)
$$
for $X$ our 3 singularity loci. In $t$-degree 0, this calculates the topological Euler characteristics of all singularity loci: 

\begin{gather*}
\chi\big(A_0(F)\big) =  \left[(2\ell+4-2\ell d-2d)(\Sigma_\ell-d^{2\ell+1})-(\ell+1){2\ell+2\choose \ell+1}(d-1)^{\ell+1}+\ell+2\right]\\\\
 \chi\big(A_0^2(F)\big) = 2(-\ell d+d-\ell-2)\cdot (\Sigma_\ell-d^{2\ell+1})+\ell{2\ell+2\choose \ell+1}(d-1)^{\ell+1}\\\\
 \chi\big(A_1(F)\big) = {2\ell+2\choose \ell+1}(d-1)^{\ell+1}
 \end{gather*}
 \begin{align*}
 \chi\big(\overline{A_0^2(F)}\big) &= \chi\big(A_0^2(F)\big) + \chi\big(A_1(F)\big)\\
 &=2(-\ell d+d-\ell-2)\cdot(\Sigma_\ell-d^{2\ell+1})+(\ell+1){2\ell+2\choose\ell+1}(d-1)^{\ell+1}.
\end{align*}
We remarked previously that the quantity $\Sigma_\ell-d^{2\ell+1}\equiv c_\ell-s_\emptyset$ is divisible by $(d-1)$. From this it is clear to see that when $d=1$,

\[
\chi\big(A_0^2(F)\big)=\chi\big(A_1(F)\big) =\chi\big(\overline{A_0^2(F)}\big)= 0
\]
and $\chi\big(A_0(F)\big) = \ell+2$, which is the Euler characteristic of $\mathbb P^{\ell+1}$. This statement is expected geometrically: When $d=1$, $F$ must be an inclusion along coordinate subspaces, which must have no $A_0^2$ or $A_1$ points. 

What is less easy to see is that, without choosing a particular $d$,  

\[
\chi\big(A_0(F)\big) + \chi\big(A_0^2(F)\big) + \chi\big(A_1(F)\big) = \ell+2
\]
but a computer algebra system can do this. This equality is much more complicated than the previous, as all 3 terms can be non-zero, and vary with $d$. Geometrically, it is a manifestation of the stability of $F$ in our chosen dimension pairs, as the domain must consist of only $A_0,A_0^2$, and $A_1$ points. 

\begin{question*}
    
Further, experiments suggest that for all $\ell$, $\chi(A_0^2(F))$ is divisible by $(d-2)$, but not $\chi\big(\overline{A_0^2(F)}\big)$. What geometric interpretation can be given for this fact?

\end{question*}

We list some values in the case of $\ell=2$, corresponding to a degree $d$ map $\mathbb P^3\to \mathbb P^5$, in Table (\ref{Euler chars}).

\begin{table}
    \centering
    \begin{tabular}{|c|c|c|c|c|c|c|}
         \hline&  $d=1$  &$d=2$  &$d=3$  &$d=4$  &$d=5$ \\\hline
         $\chi(A_0)$&  4& -16 &1224  &12304  &59084  \\\hline
         $\chi(A_0^2)$&0  &0  &-1380  &-12840  &-60360   \\\hline
         $\chi(A_1)$&  0&20  &160  &540  &1280   \\\hline
    \end{tabular}
    \caption{Euler characteristics of singularity loci for a degree $d$ map $\bb P^3\to \bb P^5$. The sum of any column is equal to 4, which is the Euler characteristic of $\bb P^3$.}
    \label{Euler chars}
\end{table}
In $t$-degree 1, we have

\begin{gather*}
\chi\big(A_0(F)\cap H\big) = \Sigma_\ell-d^{2\ell+1}+\ell+1\\\\
\chi\big(\overline{A_0^2(F)}\cap H \big) = \chi\big(A_0^2(F)\cap H\big) = \Sigma_\ell-d^{2\ell+1}\\\\
\chi\big(A_1(F)\cap H\big) = 0.
\end{gather*}
 The sum of all 3 is equal to $\ell+1$, which is the Euler characteristic of $\mathbb P^{\ell}=\mathbb P^{\ell+1}\cap H$. $A_1(F)\cap H$ is empty, so the third equality is also expected. 

As $\overline{A_0^2(F)}\cap H$ is 0-dimensional, the formula $\Sigma_\ell-d^{2\ell+1}$ simply counts the number of points. If $X\subset \mathbb P^n$ is a subvariety of dimension $r$, its intersection with $r$ generic hyperplanes, $X\cap H_1 \cap \dots \cap H_r$ is a finite collection of points, and the number of points is known as the \textit{degree} of $X$. So we have 
\[
\deg\big(\overline{A_0^2(F)}\big)=\Sigma_\ell-d^{2\ell+1}.
\]
Some values of Euler characteristics and degrees of $\overline{A_0^2(F)}$ are presented in Table (\ref{degrees}). 
\begin{table}
    \centering
    \begin{tabular}{|c|c|c|c|c|c|c|}
        \hline
         &$\ell=1$  &$\ell=2$  & $\ell=3$ \\\hline
         $d=1$&  (0,0)& (0,0) & (0,0)   \\\hline
         $d=2$&  (6,3)&  (20,10)&  (70,35)  \\\hline
         $d=3$&  $(-60,18)$&  (-1220,170)&(-18060,1610)     \\\hline
         $d=4$&  $(-402,51)$& (-12300,870) &  (-298410,14595)     \\\hline
         $d=5$&  (-1320,108)& (-59080,2860) &(-10168410,272195)       \\\hline
    \end{tabular}
    \caption{Ordered pairs $(\chi,D)$ where $\chi$ is the topological Euler characteristic and $D$ is the degree of the smooth surface $\overline{A_0^2(f)}$. }
    \label{degrees}
\end{table}

For $t$-degrees 2 and above, only $A_0(F)$ has a contribution, which simply says, for $i=2,\dots,\ell+1$,
\[
\chi\big(A_0(F)\cap H_1\cap\cdots\cap H_i\big) = (\ell+1)-(i-1)
\]
and indeed is independent of $d$. 

\subsection{Application to complete intersections}

A subvariety $X \subset \mathbb P^n$ is called a \textit{complete intersection} if it is an intersection of codimension 1 subvarieties. For example when $n=3$, 1-dimensional complete intersections are of the form $H_1\cap H_2$, for $H_i\subset \mathbb P^3$ hyperplanes. 

Though ``most'' subvarieties are not complete intersections, determining whether a given subvariety is a complete intersection is a difficult problem. 

Again considering a $\mathcal C_\ell$-map $F: \mathbb P^{\ell+1}\to \mathbb P^{2\ell+1}$, $X:=\overline{A_0^2(F)}$ is a Riemann surface. If $X$ were a complete intersection of subvarieties of degrees $(d_1,\dots,d_\ell)$, $S_1,\dots,S_\ell$, its topological Euler characteristic would be 

\[
\chi_{CI}(d_1,\dots,d_\ell) = -\prod_{i=1}^\ell d_i \left(\sum_{i=1}^\ell d_i-(\ell+2)\right)
\]
while its degree would be $\deg_{CI}(d_1,\dots,d_\ell)=d_1d_2\cdots d_\ell$.

On the other hand, we just obtained formulas for the Euler characteristic and degree of $X$, depending only on the relative dimension $\ell$ and degree of $F$, $d$, (note that there are two independent notions of degree being discussed)
\[
\chi_X(d) = (-2\ell d+2d-2\ell -4)(\Sigma_\ell-d^{2\ell+1})+(\ell+1){2\ell+2\choose \ell+1}(d-1)^{\ell+1}
\]
\[
\deg_X(d)=\Sigma_\ell-d^{2\ell+1}.
\]
If $X$ were a complete intersection, then it must be the case that $\chi_{CI}(d_1,\dots,d_\ell) = \chi_X(d)$ and $\deg_{CI}(d_1,\dots,d_\ell) = \deg_X(d)$. For concrete values of $\ell$, this system can be shown to be inconsistent.

\begin{example}
    Let us take the case $\ell=2$ $(\ell=1$ is trivial in this context), and suppose $F: \bb P^{\ell+1}\to \bb P^{2\ell+1}$ is a map of degree $d\geq 0$. Then for $X:=\overline{A_0^2(F)}$ to be a complete intersection, it must be that the quantity $\deg_X(d)=\Sigma_2-d^5$ admits a pair of complementary divisors, $d_1,d_2$, such that
    \begin{gather*}
    \chi_{CI}(d_1,d_2) = \chi_X(d).
    \end{gather*}
    In this case, the formulas found reduce to (replacing $d_1d_2$ with $\deg_X(d)$)
    \begin{gather*}
    \chi_{CI}(d_1,d_2) =-(d^5-15d^2+24d-10)\cdot (d_1+d_2-4) \\
    \chi_X(d) = -6d^6+8d^5+150d^3-444d^2+432d-140
    \end{gather*}
    and we aim to show that no pair $d_1,d_2$ can render the equation $\chi_{CI}(d_1,d_2) =\chi_X(d)$ true, by estimating (both quantities are negative), that $\chi_{CI}$'s maximal value is strictly less than $\chi_X$. Specifically, as $d$ increases past $d=1$, $\chi_X$ tends to $-\infty$ much faster than $\chi_{CI}$. Thinking of $d$ as a fixed parameter, $\chi_{CI}(d_1,d_2)$ is maximized when $d_1+d_2$ is minimized. Of course,
    \[
    d_1+d_2 = d_1 + \frac{\deg_X(d)}{d_1}
    \]
    and the first derivative test shows that $d_1+d_2$ is minimized, subject to the constraint $d_1d_2 = \deg_X(d)$, at $d_1 = d_2=\sqrt{\deg_X(d)}$. Therefore 
    \[
    \chi_{CI}^{max} = -(d^5-15d^2+24d-10)\cdot (2\cdot \sqrt{\deg_X(d)}-4)
    \]
    is the maximal quantity $\chi_{CI}(d_1,d_2)$ can achieve under the constraint on $d_1,d_2$ (and depends only on $d$). Yet we claim 
    \begin{align*}
    &\chi_X(d)-\chi^{max}_{CI}(d)\\
    &= ( 2{d}^{5}-30{d}^{2}+48d-20) \sqrt {{d}^{5}-15{d}
^{2}+24d-10}\\
&-6{d}^{6}+4{d}^{5}+150{d}^{3}-384{d}^{2}+336d-
100
    \end{align*}
    is strictly positive at integer values of $d$ except $d=1$, so the system cannot be solved. Indeed, denoting 
    \begin{gather*}
        A(d):=2d^5-30d^2+48d-20\\
        B(d):=d^5-15d^2+24d-10\\
        C(d):=-6{d}^{6}+4{d}^{5}+150{d}^{3}-384{d}^{2}+336d-100
    \end{gather*}
    Then $\chi_X(d)-\chi_{CI}^{max}(d)>0$ iff 
    \begin{gather*}
        A(d)\sqrt{B(d)}>-C(d)
    \end{gather*}
    $C(d)$ is strictly negative for integer $d>1$, and $A,\sqrt{B(d)}$ are strictly positive, so we can square both sides,
    \[
    A(d)^2B(d)>C(d)^2
    \]
    This is now a polynomial inequality which can be verified directly. 
\end{example}

\newpage
\appendix
\section{Appendix}

To complete the proof of Theorem (\ref{SSM of A02}), it must be shown that 

\begin{lemma}\label{last-step}
\begin{equation}
h_N(\alpha,\beta_1,\dots,\beta_\ell) = \left[\sum_{k=0}^NS_{(N-k)}{\ell+N-1\choose k}(-\alpha)^k\right]
\end{equation}
where
\[
S_{(m)} = \begin{vmatrix}c_1 & c_2 & \dots&c_{m}\\ 1 & c_1 & \dots&c_{m-1}\\0&1&\dots\\ & \vdots & \\
&&1&c_1\end{vmatrix}.
\]
\end{lemma}
The difficulty in proving this lemma is that on the LHS we have a symmetric polynomial in the arguments $\alpha,\beta_1,\dots,\beta_\ell$, while on the RHS, in $S_{(N-k)}$, we have symmetric polynomials in the arguments $\alpha+\beta_1,\dots,\alpha+\beta_\ell$, so there is no straightforward relation between the two. Further, the actual dependence of the RHS on $\alpha$ and $\beta$'s is buried deeply under several definitions. Unraveling these two features is the key to showing the desired identity. Note this matrix is a Toeplitz matrix, i.e. all diagonals are constant, though we do not employ any of the theory of such.

\begin{proof}
We obtained an expression for $c_k$ in the course of the proof of Theorem (\ref{SSM of A02})
\[
c_k=e_k(\alpha+\beta_1,\dots,\alpha+\beta_\ell) + (-1)^{k+1}\alpha^k+\sum_{a=1}^{k-1 } e_a(\alpha+\beta_1,\dots,\alpha+\beta_\ell)\cdot (-1)^{-(k-a)+1}\alpha^{k-a} 
\]
which, when the subscript exceeds $\ell$, reduces to 
\[
c_{\ell+k} = (-1)^{k+1}\alpha^k\beta_1\cdots\beta_\ell
\]
for $k>0$. We aim to reduce the RHS of (\ref{last-step}) to LHS, so we begin by unraveling $S_{(m)}$ in terms of $\alpha$'s and $\beta$'s:

\begin{lemma}\label{Sm determinant}
\[
S_{(m)}=\left[\sum_{i=1}^{m}2^{i-1}\cdot \alpha^i\cdot h_{m-i}(\alpha+\beta_1,\dots,\alpha+\beta_\ell)\right] + h_{m}(\alpha+\beta_1,\dots,\alpha+\beta_\ell).
\]
\end{lemma}
\begin{proof}
Using the cofactor expansion of determinant once using the first row, we obtain 
\begin{gather*}
S_{(m)}(c_1,\dots,c_m) \equiv \begin{vmatrix}c_1 & c_2 & c_3&\dots&c_{m-1} & c_{m}\\ 1 & c_1 &c_2& \dots&c_{m-2} & c_{m-1}\\0&1&c_1&\dots\\ & \vdots & \\
&&&&1&c_1\end{vmatrix}\\
=c_1\begin{vmatrix}c_1 & c_2 &c_3& \dots&c_{m-2} & c_{m-1}\\ 1 & c_1 & c_2&\dots&c_{m-3} & c_{m-2}\\0&1&c_1&\dots\\ & \vdots & \\
&&&&1&c_1\end{vmatrix}-c_2\begin{vmatrix}1 & c_2 & c_3 &\dots&c_{m-2} & c_{m-1}\\ 0& c_1 & c_2 & \dots&c_{m-2} & c_{m-3}\\0&1&c_1&\dots\\ & \vdots & \\
&&&&1&c_1\end{vmatrix} \\
+ \dots + (-1)^{m+1}c_m\begin{vmatrix}1 & c_1 & c_2 &\dots&c_{m-3} & c_{m-2}\\ 0& 1 & c_1 & \dots&c_{m-4} & c_{m-3}\\0&0&1&\dots\\ & \vdots & \\
&&&&0&1\end{vmatrix}.
\end{gather*}
The first term is exactly $c_1S_{(m-1)}$. The second may be computed by again performing a cofactor expansion, this time on the first column, since there is only one non-zero entry. The second term then becomes $-c_2\cdot S_{(m-2)}$. The $n$th term in the sum yields $(-1)^{n+1}c_nS_{(m-n)}$ after performing $n-1$ cofactor expansions. We obtain
$$
S_{(m)} = \sum_{i=1}^m(-1)^{i+1} c_i S_{(m-i)}.
$$
By strong induction we may assume 
$$
S_{(m-i)} = \left[\sum_{k=1}^{(m-i)}2^{k-1}\cdot \alpha^k \cdot h_{(m-i)-k}(\alpha+\beta_1,\dots,\alpha+\beta_\ell)\right]+h_{(m-i)}(\alpha+\beta_1,\dots,\alpha+\beta_\ell)
$$
for all $i \in \{1,2,\dots,m\}$, thus (For now, all symmetric polynomials have arguments $\alpha+\beta_1,\dots,\alpha+\beta_\ell$, so we abbreviate to simply $e_{(-)}$ and $h_{(-)}$)
\begin{gather*}
S_{(m)} = \sum_{i=1}^m(-1)^{i+1}c_i\left( \left[\sum_{k=1}^{(m-i)}2^{k-1}\cdot \alpha^k \cdot h_{(m-i)-k}\right]+h_{(m-i)}\right)\\
= \sum_{i=1}^m(-1)^{i+1}\left(e_i+(-1)^{i+1}\cdot \alpha^i + \sum_{a=1}^{i-1}e_a\cdot (-1)^{-(i-a)+1}\alpha^{i-a}\right)\cdot\left( \left[\sum_{k=1}^{(m-i)}2^{k-1}\cdot \alpha^k \cdot h_{(m-i)-k}\right]+h_{(m-i)}\right)\\\\
=\sum_{i=1}^m(-1)^{i+1}\Bigg(e_i \cdot \sum_{k=1}^{m-i}2^{k-1}\alpha^kh_{m-i-k} + e_i h_{m-i} + (-1)^{i+1}\alpha^i\cdot \sum_{k=1}^{m-i}2^{k-1}\alpha^kh_{m-i-k} + (-1)^{i+1}\alpha^i h_{m-i} \\
+ \sum_{a=1}^{i-1}e_a (-1)^{-(i-a)+1}\alpha^{i-a}\cdot \sum_{k=1}^{m-i}2^{k-1}\alpha^kh_{m-i-k} + \sum_{a=1}^{i-1}e_a(-1)^{-(i-a)+1}a^{i-a}h_{m-i}\Bigg).
\end{gather*}
Pulling the second term out of the parenthesis, then out of the sum, we get
\begin{gather*}
= \sum_{i=1}^m(-1)^{i+1} e_ih_{m-i} + \sum_{i=1}^m(-1)^{i+1}\Bigg(e_i \cdot \sum_{k=1}^{m-i}2^{k-1}\alpha^kh_{m-i-k} + (-1)^{i+1}\alpha^i\cdot \sum_{k=1}^{m-i}2^{k-1}\alpha^kh_{m-i-k} + (-1)^{i+1}\alpha^i h_{m-i} \\
+ \sum_{a=1}^{i-1}e_a (-1)^{-(i-a)+1}\alpha^{i-a}\cdot \sum_{k=1}^{m-i}2^{k-1}\alpha^kh_{m-i-k} + \sum_{a=1}^{i-1}e_a(-1)^{-(i-a)+1}\alpha^{i-a}h_{m-i}\Bigg).
\end{gather*}
The identity 
$$
\sum_{i=0}^m(-1)^ie_i h_{m-i}=0
$$
implies that the first sum above may be replaced with $h_m$, since the sum is missing the $i=0$ term:
\begin{gather*}
= h_m+ \sum_{i=1}^m(-1)^{i+1}\Bigg(e_i \cdot \sum_{k=1}^{m-i}2^{k-1}\alpha^kh_{m-i-k} + (-1)^{i+1}\alpha^i\cdot \sum_{k=1}^{m-i}2^{k-1}\alpha^kh_{m-i-k} + (-1)^{i+1}\alpha^i h_{m-i} \\
+ \sum_{a=1}^{i-1}e_a (-1)^{-(i-a)+1}\alpha^{i-a}\cdot \sum_{k=1}^{m-i}2^{k-1}\alpha^kh_{m-i-k} + \sum_{a=1}^{i-1}e_a(-1)^{-(i-a)+1}\alpha^{i-a}h_{m-i}\Bigg).
\end{gather*}
Splitting the sums into terms involving $e$'s and terms not involving $e$'s, we obtain
\begin{gather*}
= h_m+ \underline{\sum_{i=1}^m(-1)^{i+1}\Bigg(e_i \cdot \sum_{k=1}^{m-i}2^{k-1}\alpha^kh_{m-i-k} + \sum_{a=1}^{i-1}e_a (-1)^{-(i-a)+1}\alpha^{i-a}\cdot \sum_{k=1}^{m-i}2^{k-1}\alpha^kh_{m-i-k}} \\
\underline{+h_{m-i}\sum_{a=1}^{i-1}e_a(-1)^{-(i-a)+1}\alpha^{i-a}\Bigg)}\\
+\sum_{i=1}^m\left(\left[\sum_{k=1}^{m-i}2^{k-1}\alpha^{k+i}h_{m-i-k}\right]+\alpha^ih_{m-i}\right)
\end{gather*}
Now we claim that the underlined portion is equal to 0:

\begin{lemma}\label{sum=0}
$$
\sum_{i=1}^m(-1)^{i+1}\Bigg(e_i \cdot \sum_{k=1}^{m-i}2^{k-1}\alpha^kh_{m-i-k} + \sum_{a=1}^{i-1}e_a (-1)^{-(i-a)+1}\alpha^{i-a}\cdot \sum_{k=1}^{m-i}2^{k-1}\alpha^kh_{m-i-k}
$$
$$
+h_{m-i}\sum_{a=1}^{i-1}e_a(-1)^{-(i-a)+1}\alpha^{i-a}\Bigg)=0
$$
\end{lemma}

\begin{proof}
We proceed by induction on $m$: Let $P(m)$ be the proposition that the above sum is equal to 0. Checking the base case is routine, but there is one thing to note: That the sum is equal to 0 can be obtained \textit{as a polynomial in the $\alpha$, $e$'s and $h$'s. In particular, we do not need to refer to any properties of the $e$'s and $h$'s: we do not even need to assume that $h_0=1$. This implies that if we rename all the variables, the identity will still hold. This will be relevant later.}

We must show that 
$$
\sum_{i=1}^{m+1}(-1)^{i+1}\Bigg(e_i \cdot \sum_{k=1}^{m+1-i}2^{k-1}\alpha^kh_{m+1-i-k} + \sum_{a=1}^{i-1}e_a (-1)^{-(i-a)+1}\alpha^{i-a}\cdot \sum_{k=1}^{m+1-i}2^{k-1}\alpha^kh_{m+1-i-k}
$$
$$
+h_{m+1-i}\sum_{a=1}^{i-1}e_a(-1)^{-(i-a)+1}\alpha^{i-a}\Bigg)=0.
$$
Separating out the $i=m+1$ term (many of the inner sums are thus 0 because they start at $k=1$ and $a=1$), we get
\begin{gather*}
=\underbrace{(-1)^m\cdot h_0 \sum_{a=1}^me_a\cdot(-1)^{a-m}\cdot \alpha^{m+1-a}}_{i=m+1}\\
+\sum_{i=1}^{m}(-1)^{i+1}\Bigg(e_i \cdot \sum_{k=1}^{m+1-i}2^{k-1}\alpha^kh_{m+1-i-k} + \sum_{a=1}^{i-1}e_a (-1)^{-(i-a)+1}\alpha^{i-a}\cdot \sum_{k=1}^{m+1-i}2^{k-1}\alpha^kh_{m+1-i-k}\\
+h_{m+1-i}\sum_{a=1}^{i-1}e_a(-1)^{-(i-a)+1}\alpha^{i-a}\Bigg).
\end{gather*}
Then we separate out the final term of the inner sums that involve $m+1$
\begin{gather*}
= h_0 \sum_{a=1}^me_a\cdot(-1)^a\cdot \alpha^{m+1-a}\\
+\sum_{i=1}^{m}(-1)^{i+1}\Bigg(e_i \cdot \left[\sum_{k=1}^{m-i}2^{k-1}\alpha^kh_{m+1-i-k} +\underbrace{2^{m-i}\alpha^{m+1-i}h_0}_{k=m+1-i}\right]\\
+ \sum_{a=1}^{i-1}e_a (-1)^{-(i-a)+1}\alpha^{i-a}\cdot \left[\sum_{k=1}^{m-i}2^{k-1}\alpha^kh_{m+1-i-k}+\underbrace{2^{m-i}\alpha^{m+1-i}h_0}_{k=m+1-i}\right]\\
+h_{m+1-i}\sum_{a=1}^{i-1}e_a(-1)^{-(i-a)+1}\alpha^{i-a}\Bigg).
\end{gather*}
Then pull these extra $k=m+1-i$ terms all the way out of the outermost sum, upon which they pick up a sum over $i$ and an alternating sign,
\begin{gather*}
=h_0\sum_{a=1}^me_a\cdot(-1)^a\cdot \alpha^{m+1-a}\\
+\sum_{i=1}^m(-1)^{i+1}\left(e_i\cdot 2^{m-i}\alpha^{m+1-i}h_0 + \sum_{a=1}^{i-1}e_a(-1)^{-(i-a)+1}\alpha^{i-a}\cdot(2^{m-i}\alpha^{m+1-i}h_0)\right)\\
+\sum_{i=1}^{m}(-1)^{i+1}\Bigg(e_i \cdot \left[\sum_{k=1}^{m-i}2^{k-1}\alpha^kh_{m+1-i-k}\right]\\
+ \sum_{a=1}^{i-1}e_a (-1)^{-(i-a)+1}\alpha^{i-a}\cdot \left[\sum_{k=1}^{m-i}2^{k-1}\alpha^kh_{m+1-i-k}\right]\\
+h_{m+1-i}\sum_{a=1}^{i-1}e_a(-1)^{-(i-a)+1}\alpha^{i-a}\Bigg).
\end{gather*}
We want to cancel the final 3 lines in the expression above using the inductive hypothesis. However the sum in $P(m)$ involves variables $h_0,\dots,h_{m-i-1}$, while the expression above has variables $h_1,\dots,h_{m-i}$, due to the $+1$ in all $h$ subscripts. As argued before, $P(m)$ is true for any indeterminate variables, so we can still apply the inductive hypothesis to conclude the sum of the final 3 lines are 0,
\begin{gather*}
=h_0 \sum_{a=1}^me_a\cdot(-1)^a\cdot \alpha^{m+1-a}\\
+\sum_{i=1}^m(-1)^{i+1}2^{m-i}\alpha^{m+1-i}h_0\left(e_i+\sum_{a=1}^{i-1}e_a(-1)^{-(i-a)+1}\alpha^{i-a}\right)\\\\
=h_0 \sum_{a=1}^me_a\cdot(-1)^a\cdot \alpha^{m+1-a}\\
+\sum_{i=1}^m(-1)^{i+1}2^{m-i}\alpha^{m+1-i}h_0e_i\\
+\sum_{i=1}^m(-1)^{i+1}2^{m-i}\alpha^{m+1-i}h_0\sum_{a=1}^{i-1}e_a(-1)^{-(i-a)+1}\alpha^{i-a}.
\end{gather*}
Every term has a factor of $h_0$. Since we seek to show that this is equal to 0, we can factor and ignore the $h_0$ (note again that we do not need to impose that $h_0=1)$
\begin{gather*}
=\sum_{a=1}^me_a\cdot(-1)^a\cdot \alpha^{m+1-a}+\sum_{i=1}^m(-1)^{i+1}2^{m-i}\alpha^{m+1-i}e_i+\sum_{i=1}^m\sum_{a=1}^{i-1}e_a(-1)^a2^{m-i}\alpha^{m+1-a}.
\end{gather*}
In the third term, we swap the order of summation
\begin{gather*}
= \sum_{a=1}^me_a\cdot(-1)^a\cdot \alpha^{m+1-a}+\sum_{i=1}^m(-1)^{i+1}2^{m-i}\alpha^{m+1-i}e_i+\sum_{a=1}^{m-1}\sum_{i=a+1}^{m}e_a(-1)^a2^{m-i}\alpha^{m+1-a}\\\\
=\sum_{a=1}^me_a\cdot(-1)^a\cdot \alpha^{m+1-a} +\sum_{i=1}^m(-1)^{i+1}2^{m-i}\alpha^{m+1-i}e_i +\sum_{a=1}^{m-1}e_a(-1)^a \alpha^{m+1-a}\sum_{i=a+1}^{m}2^{m-i}.
\end{gather*}
The inner sum over $i$ in the third term is now a geometric series with value $2^{m-a}-1$,
\begin{gather*}
= \sum_{a=1}^me_a\cdot(-1)^a\cdot \alpha^{m+1-a}+\sum_{i=1}^m(-1)^{i+1}2^{m-i}\alpha^{m+1-i}e_i+\sum_{a=1}^{m-1}e_a(-1)^a \alpha^{m+1-a}\cdot\big(2^{m-a}-1\big)
\end{gather*}
\begin{gather*}
= \sum_{a=1}^me_a\cdot(-1)^a\cdot \alpha^{m+1-a}+\sum_{i=1}^m(-1)^{i+1}2^{m-i}\alpha^{m+1-i}e_i\\
+\sum_{a=1}^{m-1}e_a(-1)^a \alpha^{m+1-a}\cdot 2^{m-a}-\sum_{a=1}^{m-1}e_a(-1)^a \alpha^{m+1-a}.
\end{gather*}
The second and third sums are identical with opposite signs except the second sum has an extra $m$ term, so the overall contribution of sum two + sum three is just the $i=m$ term of sum two:
\begin{gather*}
= \sum_{a=1}^me_a\cdot(-1)^a\cdot \alpha^{m+1-a}+(-1)^{m+1}2^0\alpha^{1}e_m -\sum_{a=1}^{m-1}e_{a}(-1)^a \alpha^{m+1-a}.
\end{gather*}
The term that just appeared is the negative of the $a=m$ term of (the new) sum three. Since the sum comes with a - sign, it can then be absorbed 
\begin{gather*}
= \sum_{a=1}^me_{a}\cdot(-1)^a\cdot \alpha^{m+1-a} -\sum_{a=1}^{m}e_{a}(-1)^a \alpha^{m+1-a} \equiv 0.
\end{gather*}
\end{proof}
Returning to the proof of Lemma (\ref{Sm determinant}), we can write 
\begin{gather*}
S_{(m)}=h_m +\cancel{(-)} + \sum_{i=1}^m\left(\left[\sum_{k=1}^{m-i}2^{k-1}\alpha^{k+i}h_{m-i-k}\right]+\alpha^ih_{m-i}\right)\\
=h_m +\sum_{i=1}^m\sum_{k=1}^{m-i}2^{k-1}\alpha^{k+i}h_{m-i-k}+\sum_{i=1}^m\alpha^ih_{m-i}
\end{gather*}
where $\cancel{(-)}$ is the term which is shown to be equal to 0 in Lemma (\ref{sum=0}). If we reindex $j:=k+i$, 
\begin{gather*}
=h_m + \left(\sum_{j=2}^m\sum_{i=1}^{j-1}2^{j-i-1}\alpha^jh_{m-j}\right)+\sum_{i=1}^m\alpha^ih_{m-i}\\
=h_m + \left(\sum_{j=2}^m\alpha^jh_{m-j}\sum_{i=1}^{j-1}2^{j-i-1}\right)+\sum_{i=1}^m\alpha^ih_{m-i}.
\end{gather*}
The inner sum over $i$ is now a geometric series with value $2^{j-1}-1$,
\begin{gather*}
=h_m + \left(\sum_{j=2}^m\alpha^jh_{m-j}(2^{j-1}-1)\right)+\sum_{i=1}^m\alpha^ih_{m-i}\\
=h_m + \sum_{j=2}^m2^{j-1}\alpha^jh_{m-j}-\sum_{j=2}^m\alpha^jh_{m-j}+\sum_{i=1}^m\alpha^ih_{m-i}\\
=h_m + \sum_{j=2}^m2^{j-1}\alpha^jh_{m-j}+\alpha^1h_{m-1}\\
=h_m + \sum_{j=1}^m2^{j-1}\alpha^jh_{m-j}
\end{gather*}
as required.
\end{proof}
Returning to the proof of Lemma (\ref{last-step}), we now need to show 
\begin{equation}\label{last-eqn}
\sum_{k=0}^N\left(\sum_{i=1}^{N-k}2^{i-1}\cdot \alpha^i\cdot h_{N-k-i}+h_{N-k}\right){\ell+N-1\choose k}(-1)^{k+1}\alpha^k = h_N(\alpha,\beta_1,\dots,\beta_\ell).
\end{equation}
Again any symmetric polynomials without arguments explicitly listed are understood to have arguments $\alpha+\beta_1,\dots,\alpha+\beta_\ell$, such as on the LHS of equation (\ref{last-eqn}). 
The following Lemma is a straightforward generalization of Lemma 3.5 of (\cite{thesis}), where they prove the case of $\alpha=1$ using the technology of the ``category of sequences increasing to $m$'':
\begin{lemma}\label{lemma-thesis}
\[
h_k(\alpha+\beta_1,\dots,\alpha+\beta_\ell) = \sum_{j=0}^k{\ell+k-1\choose k-j}\alpha^{k-j}h_j(\beta_1,\dots,\beta_\ell).
\]
\end{lemma}
From which we obtain the corollary
\begin{corollary}\label{thesis-corollary}
\[
h_k(\beta_1,\dots,\beta_\ell) = \sum_{j=0}^k{\ell+k-1\choose k-j}(-1)^{k-j}\alpha^{k-j}h_j(\alpha+\beta_1,\dots,\alpha+\beta_\ell).
\]
\end{corollary}
\begin{proof}
Take the equation in the above proposition and first send $b_i\mapsto b_i-\alpha$ while leaving $\alpha$ unchanged, then send $\alpha\mapsto -\alpha$. 
\end{proof}
Now we can begin to justify (\ref{last-eqn}). The RHS may be rewritten using Corollary (\ref{thesis-corollary}),
\[
\text{RHS } (\ref{last-eqn}) = \sum_{k=0}^N\alpha^k \left(\sum_{i=0}^{N-k}  {\ell+(N-k)-1\choose (N-k)-i}(-1)^{(N-k)-i}\alpha^{(N-k)-i}h_i(\alpha+\beta_1,\dots,\alpha+\beta_\ell)      \right).
\]
On the inner sum, we reverse the order of summation, swapping $i \leftrightarrow (N-k)-i$,
$$
= \sum_{k=0}^N\alpha^k \left(\sum_{i=0}^{N-k}  {\ell+(N-k)-1\choose i}(-1)^i\alpha^ih_{N-k-i}      \right).
$$
We swap the order of the double sum,
$$
=\sum_{i=0}^N\sum_{k=0}^{N-i}\alpha^k{\ell+N-k-1\choose i}(-1)^i\alpha^ih_{N-k-i}
$$
$$
=\sum_{i=0}^N(-1)^i\alpha^i\sum_{k=0}^{N-i}{\ell+N-k-1\choose i}\alpha^kh_{N-k-i}
$$
$$
=\sum_{k=0}^N(-1)^k\alpha^k\sum_{i=0}^{N-k}{\ell+N-i-1\choose k}\alpha^ih_{N-k-i}.
$$
Next we would like to justify the equality
\begin{equation}\label{last-RHS}
\sum_{k=0}^N(-1)^k\alpha^k\sum_{i=0}^{N-k}{\ell+N-i-1\choose k}\alpha^ih_{N-k-i}
\end{equation}
\begin{equation}\label{last-LHS}
=\sum_{k=0}^N\left[{\ell+N-1\choose k}(-1)^{k+1}\alpha^k\left(\sum_{i=1}^{N-k}2^{i-1}\cdot \alpha^i\cdot h_{N-k-i}+h_{N-k}\right)\right].
\end{equation}
as these are the right and left hand sides of (\ref{last-eqn}), respectively. Surely there are more efficient ways to show (\ref{last-RHS})=(\ref{last-LHS}), but here is one way, which may also be of independent interest to combinatorialists.

The monomials appearing in both equations are of the form $\alpha^Kh_{N-K}$, so if we compute the coefficient of that monomial for each equation, and show they agree, then we will be done. Computing \textit{an} expression for the coefficient in each equation is not difficult. In equation (\ref{last-RHS}) the coefficient of $\alpha^Kh_{N-K}$ is 
$$
[a^Kh_{N-K}](\ref{last-RHS}) = \sum_{i=0}^K(-1)^i{\ell+N-1-K+i\choose i}. 
$$
In equation (\ref{last-LHS}), the coefficient of $\alpha^Kh_{N-K}$ is 
$$
[\alpha^Kh_{N-K}](\ref{last-LHS}) = \left(\sum_{i=0}^{K-1} 2^{K-1-i}(-1)^i{\ell+N-1\choose i}\right)+(-1)^{K}{\ell+N-1\choose K}.
$$
Finally we claim
\begin{theorem}
For $r \in \bb N$ and $c\in [r]$ (so that the pair $(r,c)$ represents a point on Pascal's triangle), we have 
$$
\sum_{i=0}^{c-1}{r-c+i\choose i}(-1)^i = \sum_{i=0}^{c-1}2^{c-1-i}(-1)^i{r\choose i}.
$$
Adding $(-1)^c{r\choose i}$ to both sides yields the equivalent formulation
$$
\sum_{i=0}^c{r-c+i\choose i}(-1)^i = \left[\sum_{i=0}^{c-1}2^{c-1-i}(-1)^i{r\choose i}\right] + (-1)^c {r\choose c}.
$$
\end{theorem}
In words: the alternating diagonal sum approaching $(r,c)$ is equal to a weighted, alternating row sum approaching $(r,c)$. If we omit the alternating sign, the positive diagonal sum's answer is a classical combinatorics result known as the \textit{hockey stick identity}.

Note that, upon proving this theorem, we may apply it to justify the equation $(\ref{last-RHS})=(\ref{last-LHS})$ by setting $r=\ell+N-1$ and $c=K$ in the second formulation, so this proof is the final step. First, an example.
\begin{example} Choosing $r=7,c=4$, 
$$
\begin{array}{ccccccccccccccccc}
& & & & & & & & \binom{0}{0} & & & & & & & \\ 
& & & & & & & \binom{1}{0} & & \binom{1}{1} & & & & & & \\ 
& & & & & & \binom{2}{0} & & \binom{2}{1} & & \binom{2}{2} & & & & & \\ 
& & & & & \underline{\binom{3}{0}} & & \binom{3}{1} & & \binom{3}{2} & & \binom{3}{3} & & & & \\ 
& & & & \binom{4}{0} & & \underline{\binom{4}{1}} & & \binom{4}{2} & & \binom{4}{3} & & \binom{4}{4} & & & \\ 
 && & \binom{5}{0} & & \binom{5}{1} & & \underline{\binom{5}{2}} & & \binom{5}{3} & & \binom{5}{4} & & \binom{5}{5} & \\ 
 && \binom{6}{0} & & \binom{6}{1} & & \binom{6}{2} & & \underline{\binom{6}{3}} & & \binom{6}{4} & & \binom{6}{5} & & \binom{6}{6} \\ 
&\underline{\binom{7}{0}} & & \underline{\binom{7}{1}} & & \underline{\binom{7}{2}} & & \underline{\binom{7}{3}} & & \boxed{\binom{7}{4}} & & \binom{7}{5} & & \binom{7}{6} & & \binom{7}{7} 
\end{array}
$$
the alternating hockey stick sum (diagonal line towards ${7\choose 4}$ from the left) is 
$$
{3\choose 0}-{4\choose 1}+{5\choose 2}-{6\choose 3} = -13
$$
and the weighted alternating row sum (left of row containing ${7\choose 4}$) is 
$$
8{7\choose 0}-4{7\choose 1}+2{7\choose 2}-1{7\choose 3} = -13.
$$
Note that this is not a term-by-term equality. 
\end{example}
\begin{proof}
We will prove the first formulation and proceed by choosing an arbitrary diagonal line of the form $L_k:={r\choose r-k}$. If we prove the theorem for every $r\in \bb N$ in the line $L_k$, for arbitrary $k$, the theorem is proven. We fix arbitrary $k$ for the rest of the proof and induct on $r$: Let $P(r)$ be the proposition that the equality above is true when approaching the point ${r\choose r-k}$, i.e. setting $c=r-k$:
$$
P(r) = ``\sum_{i=0}^{r-k-1}{r-(r-k)+i\choose i}(-1)^i = \sum_{i=0}^{(r-k)-1}2^{(r-k)-1-i}(-1)^i{r\choose i}''
$$
$$
=``\sum_{i=0}^{r-k-1}{k+i\choose i}(-1)^i = \sum_{i=0}^{r-k-1}2^{r-k-1-i}(-1)^i{r\choose i}''.
$$
(This is the same as the original formulation of the theorem in our $r,k$ coordinates.) Assume $P(r-1)$ is true, so that 
$$
\sum_{i=0}^{r-k-2}{k+i\choose i}(-1)^i = \sum_{i=0}^{r-k-2}2^{r-k-2-i}(-1)^i{r-1\choose i}
$$
Then we must show $P(r)$ (already written above) is true. We start with the RHS of $P(r)$
\begin{gather*}
\sum_{i=0}^{r-k-1}2^{r-k-1-i}(-1)^i{r\choose i}.
\end{gather*}
Applying Pascal's relation, we obtain
\begin{gather*}
=\sum_{i=0}^{r-k-1}2^{r-k-1-i}(-1)^i\left[  {r-1\choose i} + {r-1\choose i-1}\right]\\
=\sum_{i=0}^{r-k-1}2^{r-k-1-i}(-1)^i{r-1\choose i} + \sum_{i=0}^{r-k-1}2^{r-k-1-i}(-1)^i{r-1\choose i-1}.
\end{gather*}
In the second term we send $i\mapsto i+1$,
\begin{gather*}
=\sum_{i=0}^{r-k-1}2^{r-k-1-i}(-1)^i{r-1\choose i} + \sum_{i=-1}^{r-k-2}2^{r-k-2-i}(-1)^{i+1}{r-1\choose i}.
\end{gather*}
In the second sum the $i=-1$ term is 0 so we may remove it, and separate out the $i=r-k-1$ term from the first sum 
\begin{gather*}
=\left[\sum_{i=0}^{r-k-2}2^{r-k-1-i}(-1)^i{r-1\choose i}+(-1)^{r-k-1}{r-1\choose r-k-1}\right] + \sum_{i=0}^{r-k-2}2^{r-k-2-i}(-1)^{i+1}{r-1\choose i}\\
=\left[\sum_{i=0}^{r-k-2}2^{r-k-1-i}(-1)^i{r-1\choose i}+ \sum_{i=0}^{r-k-2}2^{r-k-2-i}(-1)^{i+1}{r-1\choose i}\right] + (-1)^{r-k-1}{r-1\choose r-k-1}\\
=\left[\sum_{i=0}^{r-k-2}{r-1\choose i}(-1)^i\big(2^{r-k-1-i}-2^{r-k-2-i}\big)\right] + (-1)^{r-k-1}{r-1\choose r-k-1}\\
=\left[\sum_{i=0}^{r-k-2}{r-1\choose i}(-1)^i2^{r-k-2-i}\right]+(-1)^{r-k-1}{r-1\choose r-k-1}.
\end{gather*}
By inductive hypothesis we can replace the sum,
\begin{gather*}
=\sum_{i=0}^{r-k-2}{k+i\choose i}(-1)^i + (-1)^{r-k-1}{r-1\choose r-k-1}\\
=\sum_{i=0}^{r-k-1}{k+i\choose i}(-1)^i.
\end{gather*}
As required. 
\end{proof}
This completes the proof. 
\end{proof}

\printbibliography
\noindent
Reese Lance\\
University of North Carolina, Chapel Hill, North Carolina\\
Email address: \texttt{rlance@unc.edu} \\
\end{document}